\documentclass[12pt]{article}
\usepackage{amsfonts}
\usepackage{amssymb}
\usepackage{mathrsfs}
\usepackage{srcltx}
\usepackage{amsmath}
\usepackage{amsthm}
\usepackage{amstext}
\usepackage{amsopn}
\usepackage{graphicx}
\usepackage{bm}
\usepackage{color}
\newtheorem{theorem}{Theorem}[section]
\newtheorem{lemma}[theorem]{Lemma}
\newtheorem{proposition}[theorem]{Proposition}

\theoremstyle{definition}

\theoremstyle{remark}
\newtheorem{remark}[theorem]{Remark}

\numberwithin{equation}{section}

\newcommand{\ba}{\begin{array}}
\newcommand{\ea}{\end{array}}
\newcommand{\f}{\frac}

\newcommand{\la}{\lambda}

\newcommand{\ds}{\displaystyle}

\begin{document}
\date{}
\title{ \bf\large{Stability and bifurcation in a reaction-diffusion-advection predator-prey model}\thanks{This research is supported by Shandong Provincial Natural Science Foundation of China (No. ZR2020YQ01) and the National Natural Science Foundation of China (No. 12171117).}}
\author{Yihuan Sun\textsuperscript{1},\ \ Shanshan Chen\textsuperscript{2}\footnote{Corresponding Author, Email: chenss@hit.edu.cn}\ \
 \\
{\small \textsuperscript{1} School of Mathematics, Harbin Institute of Technology,\hfill{\ }}\\
\ \ {\small  Harbin, Heilongjiang, 150001, P.R.China.\hfill{\ }}\\
{\small \textsuperscript{2} Department of Mathematics, Harbin Institute of Technology,\hfill{\ }}\\
\ \ {\small Weihai, Shandong, 264209, P.R.China.\hfill{\ }}\\
}
\maketitle

\begin{abstract}
A reaction-diffusion-advection predator-prey model with Holling type-II predator
functional response is considered.
We show the stability/instability of the positive steady state and the existence of a Hopf bifurcation when the diffusion and advection rates are large.
Moreover, we show that advection rate can affect not only  the occurrence of Hopf bifurcations but also the values of Hopf bifurcations.

\end{abstract}

\section{Introduction}

The influence of environmental heterogeneity on population dynamics has been studied extensively. For example,
environmental heterogeneity can increase the total population size for a single species \cite{LouY2006}.
For two competing species in heterogeneous environments, if they are identical except dispersal
rates, then the slower diffuser wins \cite{Dockery1998}, whereas they can coexist in homogeneous environments.
The global dynamics for the weak competition case was investigated in \cite{LAMNi2012SIMA,LouY2006}, and it was completely classified in \cite{HeNi2016}. The heterogeneity of environments can also induce complex patterns for predator-prey interaction models, see \cite{DuYi2004,DuShi2007,LiWu2018} and references therein.

 In heterogeneous environments, the population may have a tendency to move up or down along the gradient of the environments, which is referred to as a \lq\lq advection'' term \cite{BF1995}. That is, the random diffusion term $d\Delta u$ is replaced by
 \begin{equation}\label{advectionop}
 d\Delta u-\alpha \nabla  \cdot \left( u\nabla m(x) \right),
 \end{equation}
 where $d$ is the diffusion rate, $\alpha$ is the advection rate, and $m(x)$ represents the heterogeneity of environment. The effect of advection as \eqref{advectionop} on population dynamics has been studied extensively for single and two competing species, see, e.g., \cite{Averill2017,BF1995,CantrellCL2006,Cantrell2007,ChenLou2008,CosnerLou2003,LAM2011,LamLou2014,LouS2015,ZhouXiao2018}. There also exists another kind of advection for species in streams, and the random diffusion term $du_{xx}$ is now replaced by
\begin{equation}\label{advectionop2}
d u_{xx}-\alpha u_x,
 \end{equation}
where $\alpha u_x$ represents the unidirectional flow from the upstream end to the downstream end. It is shown that, if the two competing species in streams are identical except dispersal rates, then the faster diffuser wins \cite{LouLu2014,LouZhaoZhou2019,LouZhou2015,Zhou2016}.
In \cite{LouNi2022,NieWang2020,TangChen2022,WangNie2022,XinLi2022}, the authors showed the effect of advection as \eqref{advectionop2} on the persistence of the predator and prey. One can also refer to \cite{HuangJin,jin2011seasonal,jin2019population,lam2016emergence,lutscher2006effects,lutscher2007spatial,lutscher2005effect,JinLewis,TangChenSIAM2021} and references therein for population dynamics in streams.

As is well known, periodic solutions occur commonly for predator-prey models \cite{May1972}, and Hopf bifurcation is a mechanism to induce these periodic solutions. For diffusive predator-prey models in homogeneous environments, Hopf bifurcations can be investigated following the framework of \cite{Hassard1981,Yi2009}, see also \cite{DongLi2017,ShiRuan2015,WangJ2011,WangMX2008} and references therein.
A natural question is how advection affects Hopf bifurcations for predator-prey models in heterogeneous environments.

In this paper, we aim to give an initial exploration for this question,
and investigate the effect of advection as \eqref{advectionop} on Hopf bifurcations for the following predator-prey model:
\begin{equation}\label{m1}
\begin{cases}
\ds u_t= \nabla  \cdot \left[d_1 \nabla u -\alpha_1 u \nabla m \right] + u\left(m(x) - u \right)- \frac{uv}{1+u},&x \in \Omega,\;t > 0,\\
\ds v_t = d_2 \Delta v -rv + \frac{luv}{1+u}, &x \in  \Omega,\;t > 0,\\
d_1 \partial_n u -\alpha_1 u \partial_n m =0,\;\;\partial_n v = 0, & x \in \partial \Omega,\;t > 0,\\
u(0,x) = {u_0}(x) \ge 0,v(0,x) = {v_0}(x) \ge 0,&x \in \Omega.
\end{cases}
\end{equation}
Here $\Omega$ is a bounded domain in $\mathbb R^N\;(1\le N\le 3)$ with a smooth boundary $\partial \Omega$; $n$ is the outward unit normal vector on $\partial \Omega$, and no-flux boundary conditions are imposed; $u(x,t)$ and $v(x,t)$ denote the population densities of the prey and predator at location $x$ and time $t$, respectively; $d_1,d_2>0$ are the diffusion rates; $\alpha_1\ge0$ is the advection rate; $l>0$ is the conversion rate; $r>0$ is the death rate of the predator; and the function $u/(1 + u)$ denotes the Holling type-II functional response
of the predator to the prey density.
The function $m(x)$ represents the intrinsic growth rate of the prey, which depends on the spatial environment.

Throughout the paper, we impose the following assumption:
\begin{enumerate}
  \item[$({\bf H}_1)$] $m(x)\in C^2 (\overline \Omega)$, $m(x)\ge(\not\equiv)0$ in $\overline \Omega$, and $m(x)$ is non-constant.
  \item[$({\bf H}_2)$] $\ds\frac{d_2}{d_1}=\theta>0$ and $\ds\frac{\alpha_1}{d_1}=\alpha\geq0$.
\end{enumerate}
Here $({\bf H}_2)$ is a mathematically technical condition, and it means that the dispersal and advection rates of the prey and predator are proportional.
Then letting $\tilde u=e^{-\alpha m(x)} u,\tilde t=d_1 t$, denoting $\la= 1/ d_1$, and dropping the tilde sign, model \eqref{m1} can be transformed to the following model:
\begin{equation}\label{m4}
\begin{cases}
\ds u_t=  e^{-\alpha m(x)} \nabla  \cdot \left[e^{\alpha m(x)}\nabla u \right] + \la u\left(m(x) - e^{\alpha m(x)} u - \frac{v}{1+e^{\alpha m(x)} u}\right),&x \in \Omega,\;t > 0,\\
\ds v_t =  \theta \Delta v + \la v\left(-r + \frac{l e^{\alpha m(x)} u}{1+e^{\alpha m(x)} u}\right),&x \in \Omega,\;t > 0,\\
\partial_n u = \partial_n v = 0, & x \in \partial \Omega,\;t>0,\\
u(0,x) = {u_0}(x) \ge 0,v(0,x) = {v_0}(x) \ge 0,&x \in \Omega,
\end{cases}
\end{equation}
where $\theta$ and $\alpha$ are defined in assumption $({\bf H}_2)$.

We remark that model \eqref{m1} with $\alpha_1=0$ (or respectively, model \eqref{m4} with $\alpha=0$) was investigated in \cite{Lou2017,Wangzhang2017}, and they showed that the heterogeneity of the environment can influence the local dynamics, and multiple positive steady states can bifurcate from the semi-trivial steady state by
 using $d_1,d_2$ (or respectively, $\la$) as the bifurcation parameters. In this paper, we consider model \eqref{m4} for the case that $\alpha\ne0$ and $0<\la\ll 1$.
We show that when $0<\la\ll1$, the local dynamics of model \eqref{m4} is similar to the following \lq\lq weighted'' ODEs:
\begin{equation}\label{wODEs}
\begin{cases}
\ds  u_t\int_\Omega e^{\alpha m(x)}dx=u\left(\int_{\Omega}e^{\alpha m(x)}m(x)dx-u\int_{\Omega}e^{2\alpha m(x)}dx\right)-v \int_{\Omega}\frac{e^{\alpha m(x)} u}{1+e^{\alpha m(x)} u}dx,\\
\ds v_t =-r v+\frac{l v}{|\Omega|}\int_{\Omega}\frac{ e^{\alpha m(x)} u}{1+e^{\alpha m(x)} u}dx.\\
\end{cases}
\end{equation}
A direct computation implies that model \eqref{wODEs} admits a unique positive equilibrium $(c_{0l},q_{0l})$ if and only if $l>\tilde l$, where $(c_{0l},q_{0l})$ and $\tilde l$ are defined in Lemma \ref{l1}. From the proof of Lemma \ref{l5}, one can obtain the
local dynamics model \eqref{wODEs} as follows:
\begin{enumerate}
\item [$\rm{(i)}$] If $\mathcal T(\alpha)<0$, then the positive equilibrium $(c_{0l},q_{0l})$ of model \eqref{wODEs} is stable for $l>\tilde l$;
\item [$\rm{(ii)}$] If  $\mathcal T(\alpha)>0$, then there exists $l_0>\tilde l$ such that $(c_{0l},q_{0l})$ is stable for $\tilde l<l<l_0$ and unstable for $l>l_0$, and model \eqref{wODEs} undergoes a Hopf bifurcation when $l=l_0$.
\end{enumerate}
Here $\mathcal T(\alpha)$ and $l_0$ are defined in Lemma \ref{suppst}.
Similarly,
model \eqref{m4} admits a unique positive steady state $\left(u^{(\la,l)},v^{(\la,l)}\right)$ for $(l,\la)\in[\tilde l+\epsilon,1/\epsilon]\times (0,\delta_\epsilon]$ with $0<\epsilon\ll 1$, where $\delta_\epsilon$  depends on $\epsilon$ (Theorem \ref{thglobal}), and admits similar local dynamics as model \eqref{wODEs} when $l\in[\tilde l+\epsilon,1/\epsilon]$ and $0<\la\ll1$ (Theorem \ref{ed}), see also Fig. \ref{fig11}.
Moreover, we show that the sign of $\mathcal T(\alpha)$ is key
to guarantee the existence of a Hopf bifurcation curve for model \eqref{m4}. We obtain that if $\int_{\Omega}(m(x)-1)dx<0$ and $\{x\in\Omega:m(x)>1\}\ne\emptyset$, then there exists $\alpha_*>0$ such that $\mathcal T(\alpha_*)=0$, $\mathcal T(\alpha)<0$ for $0\le\alpha<\alpha_*$, and $\mathcal T(\alpha)>0$ for $\alpha>\alpha_*$ (Theorem \ref{lHopf}). Therefore, the advection rate affects the occurrence of Hopf bifurcations (Proposition \ref{effad}). Moreover, we find that the advection rate can also  affect the values of Hopf bifurcations (Proposition \ref{last}).
\begin{figure}[htbp]
\centering
\includegraphics[width=0.75\textwidth]{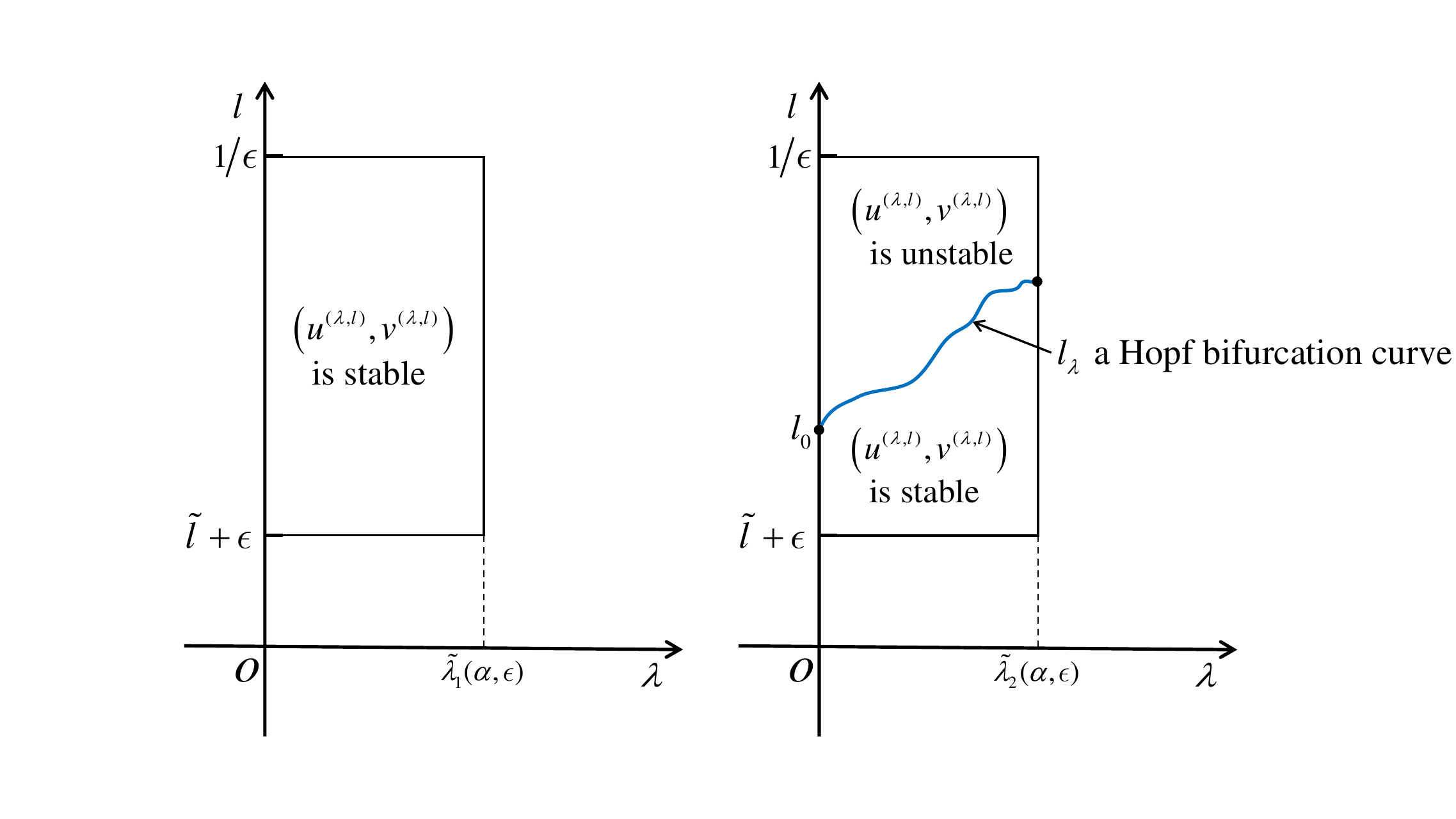}
\caption{Local dynamics of model \eqref{m4} for $(l,\la)\in[\tilde l+\epsilon,1/\epsilon]\times (0,\tilde \la_j(\alpha,\epsilon))$ with $0<\epsilon\ll1$. Here $\tilde\la_j(\alpha,\epsilon)$ means that $\tilde\la_j$ depends on $\alpha$ and $\epsilon$ for $j=1,2$. (Left): $\mathcal T(\alpha)<0$; (Right): $\mathcal T(\alpha)>0$. \label{fig11}}
\end{figure}

For simplicity, we list some notations for later use. We denote
\begin{equation*}
X=\left\{ \left.u\in H^2(\Omega ) \right|  \partial_n u = 0 \right\}\;\;\text{and} \;\; Y=L^2(\Omega).
\end{equation*}
Denote the complexification of a real linear space $Z$ by
$$ Z_{\mathbb C}:= Z\oplus{\rm i} Z=\{x_1 + {\rm i}x_2 |x_1, x_2 \in  Z\},$$
the kernel and range of a linear operator $T$ by $\mathcal N (T)$ and $\mathcal R (T)$, respectively.
For $ Y_{\mathbb C}$, we choose the standard inner product $\langle  u, v  \rangle=\int_\Omega{\overline u(x)v(x)}dx,$
and  the norm is defined by $\|  u\|_2={\langle  u, u  \rangle}^{\frac{1}{2}}.$

The rest of the paper is organized as follows. In Section 2, we show the existence and uniqueness
of the positive steady state for a range of parameters, see the rectangular region in Fig. \ref{fig11}.
In Section 3, we obtain the local dynamics of model \eqref{m4} when $(l,\la)$ is in the above rectangular region.
In Section 4, we show the effect of advection on Hopf bifurcations.

\section{Positive steady states}
In this section, we consider the positive steady states of model \eqref{m4}, which
satisfy the following system:
\begin{equation}\label{equi1}
\begin{cases}
\ds-\nabla  \cdot \left[e^{\alpha m(x)}\nabla u \right] =  \la e^{\alpha m(x)} u\left(m(x) - e^{\alpha m(x)} u - \frac{v}{1+e^{\alpha m(x)} u}\right),&x \in \Omega,\\
\ds-\theta\Delta v =  \la v\left(-r + \frac{l e^{\alpha m(x)} u}{1+e^{\alpha m(x)} u}\right),&x \in \Omega, \\
\partial_n u = \partial_n v = 0, & x \in \partial \Omega.
\end{cases}
\end{equation}
Denote
\begin{equation}\label{L}
L:=\nabla  \cdot \left[e^{\alpha m(x)}\nabla \right],
\end{equation}
and we have the following decompositions:
\begin{equation}\label{decom}
\begin{split}
&X=\mathcal N (\Delta)\oplus  X_1=\mathcal N (L)\oplus X_1,\\
&Y=\mathcal N (\Delta)\oplus Y_1=\mathcal N (L)\oplus Y_1,
\end{split}
\end{equation}
where
\begin{equation}\label{X1Y1}
\begin{split}
&X_1=\left\{ y\in X:\int_\Omega{y(x)}dx=0\right\},\\
&Y_1=\mathcal R(\Delta)=\mathcal R(L)=\left\{ y\in Y:\int_\Omega{y(x)}dx=0\right\}.
\end{split}
\end{equation}
Let
\begin{equation}\label{uv}
\begin{split}
&u= c+ \xi,\;\;\text{where}\;\;c=\ds\f{1}{|\Omega|}\int_{\Omega}udx\in\mathbb R,\;\xi\in X_1,\\
&v=q+ \eta,\;\;\text{where}\;\;q=\ds\f{1}{|\Omega|}\int_{\Omega}vdx\in\mathbb R,\;\eta\in X_1.\\
\end{split}
\end{equation}
Then substituting \eqref{uv} into \eqref{equi1}, we see that $(u,v)$ (defined in \eqref{uv}) is a solution of \eqref{equi1} if and only if $(c,q,\xi,\eta )\in \mathbb R^2\times X_1^2$ solves
\begin{equation}\label{F}
 F(c,q,\xi,\eta,l,\la)=(f_1,f_2,f_3,f_4)^T= 0,
\end{equation}
where
$F(c,q,\xi,\eta,l,\la):\mathbb R^2\times X_1^2\times \mathbb R^2\to \left(\mathbb R\times Y_1\right)^2$,
and
\begin{equation}\label{fi}
\begin{cases}
f_1(c,q,\xi,\eta,l,\la):=&\ds \int_\Omega  { e^{\alpha m(x)}( c+ \xi)\left(m(x) - e^{\alpha m(x)}( c+ \xi)- \frac{q+ \eta}{1+e^{\alpha m(x)}( c+ \xi)}\right) } dx,\\
f_2(c,q,\xi,\eta,l,\la):=&L \xi-\ds\frac{\la}{{\left| \Omega  \right|}} f_1\\
&+\ds\la e^{\alpha m(x)}( c+ \xi)\left(m(x) - e^{\alpha m(x)}( c+ \xi)- \frac{q+ \eta}{1+e^{\alpha m(x)}( c+ \xi)}\right),\\
f_3(c,q,\xi,\eta,l,\la):=&\ds \int_\Omega  { (q+ \eta)\left(-r + \frac{l e^{\alpha m(x)} ( c+ \xi)}{1+e^{\alpha m(x)}( c+ \xi)}\right)}dx ,\\
f_4(c,q,\xi,\eta,l,\la):=&\ds\theta\Delta \eta+\la(q+ \eta)\left(-r + \frac{l e^{\alpha m(x)} ( c+ \xi)}{1+e^{\alpha m(x)} ( c+ \xi)}\right)-\frac{\la}{{\left| \Omega  \right|}} f_3.\\
\end{cases}
\end{equation}

We first solve $F(c,q,\xi,\eta,l,\la)=0$ for $\la=0$.

\begin{lemma}\label{l1}
Suppose that $\la=0$, and let
\begin{equation}\label{l_*}
\tilde c=\ds\f{\int_\Omega e^{\alpha m(x)}m(x) dx}{\int_\Omega e^{2\alpha m(x)}dx}\;\;\text{and}\;\; \tilde l=\ds\frac{r |\Omega|}{\int_\Omega \frac{ \tilde c e^{\alpha m(x)} }{1+\tilde c e^{\alpha m(x)} }dx}.
\end{equation}
Then, for any $l>0$, $ F(c,q,\xi,\eta,l,\la)= 0$ has three solutions: $(0,0,0,0)$, $(\tilde c,0,0,0)$ and $(c_{0l},q_{0l},0,0)$, where $(c_{0l},q_{0l})$ satisfies
\begin{equation}\label{c0l}
\begin{split}
&\int_\Omega  \frac{c_{0l}e^{\alpha m(x)}}{1+ c_{0l} e^{\alpha m(x)} }dx=\frac{r}{l} |\Omega|,\\
&q _{0l}=\frac{l c_{0l}}{r |\Omega|} \int_\Omega e^{\alpha m(x)}  \left(m(x)- c_{0l} e^{\alpha m(x)} \right) dx.
\end{split}
\end{equation}
Moreover,  $c_{0l},q_{0l}>0$ if and only if $l>\tilde l$.
\end{lemma}
\begin{proof}
Substituting $\la =0$ into $f_2=0$ and $f_4=0$, respectively, we have $ \xi= 0$ and $ \eta= 0$. Then
substituting $\xi=\eta=0$ into $f_1=0$ and $f_3=0$, respectively, we have
\begin{equation}\label{solcq}
\begin{split}
 &c \int_\Omega  { e^{\alpha m(x)} \left(m(x) - c e^{\alpha m(x)}- \frac{q}{1+ce^{\alpha m(x)} }\right) } dx=0,\\
 &\int_\Omega  { q\left(-r + \frac{l ce^{\alpha m(x)} }{1+ce^{\alpha m(x)} }\right)}dx=0.
 \end{split}
\end{equation}
Therefore, \eqref{solcq} has three solutions:
$(0,0)$, $(\tilde c,0)$, $(c_{0l},q_{0l})$,
where $\tilde c$ is defined in \eqref{l_*}, and $(c_{0l},q_{0l})$ satisfies
\begin{equation}\label{solcq2}
\begin{split}
 &\int_\Omega  { e^{\alpha m(x)} \left(m(x) - c_{0l} e^{\alpha m(x)}- \frac{q_{0l}}{1+c_{0l}e^{\alpha m(x)} }\right) } dx=0,\\
 &\int_\Omega  {\left(-r + \frac{l c_{0l}e^{\alpha m(x)} }{1+c_{0l}e^{\alpha m(x)} }\right)}dx=0.
 \end{split}
\end{equation}
A direct computation implies that $c_{0l}$ and $q_{0l}$ satisfy \eqref{c0l}. By the second equation of \eqref{c0l}, we see that
$c_{0l},q_{0l}>0$ if and only if $0<c_{0l}<\tilde c$. It follows from the first equation of \eqref{c0l} that
\begin{equation*}
\ds\f{dc_{0l}}{dl}<0 \;\;\text{and}\;\;\lim_{l\to\infty} c_{0l}=0.
\end{equation*}
Then we obtain that  $0<c_{0l}<\tilde c$ if and only if $l>\tilde l$, where $\tilde l$ is defined in \eqref{l_*}. This completes the proof.
\end{proof}

We remark that $\tilde l$ is the critical value for the successful invasion of the predator for model \eqref{m4} with $0<\la\ll1$ (or respectively, \eqref{wODEs}). In the following we will consider the monotonicity of $\tilde l$ with respect to $\alpha$ and show the effect of advection rate on the invasion of the predator.
\begin{proposition}\label{cl}
Let $\tilde l(\alpha)$ be defined in \eqref{l_*}. Then
\begin{equation}\label{estil}
\ds \tilde l(\alpha)\geq\f{r(  |\Omega|+V(\alpha) )} {V(\alpha)} \;\;\text{for all}\;\; \alpha>0,
\end{equation}
where
\begin{equation}\label{Valpha}
\ds V(\alpha):=\f{\int_\Omega e^{\alpha m(x)}m(x) dx \int_\Omega e^{\alpha m(x)}dx } {\int_\Omega e^{2 \alpha m(x)} dx }.
\end{equation}
Moreover, the following statements hold:
\begin{enumerate}
\item [${\rm (i)}$] $\tilde l'(\alpha)|_{\alpha=0} < 0;$
\item [${\rm (ii)}$] If $\ds\lim_{\alpha\to\infty} V(\alpha)=0$, then $\ds\lim_{\alpha\to+\infty} \tilde l(\alpha)=\infty$. Especially, if $\Omega=(0,1)$ and $m'(x)>0$ (or respectively, $m'(x)<0$) for all $x\in [0,1]$, then $\ds\lim_{\alpha\to+\infty} V(\alpha)=0$.
\end{enumerate}

\end{proposition}
\begin{proof}
Since function $\ds \f{x}{1+x}$ is concave, it follows from the Jensen's inequality that
$$
\f{1}{|\Omega|}\int_\Omega \f{ \tilde c(\alpha) e^{\alpha m(x)} }{ 1+ \tilde c(\alpha) e^{\alpha m(x)}}  dx\leq\f{\f{1}{|\Omega|}\int_\Omega \tilde c(\alpha) e^{\alpha m(x)} dx}   {1+\f{1}{|\Omega|}\int_\Omega \tilde c (\alpha) e^{\alpha m(x)} dx},
$$
where $\tilde c(\alpha)$ is defined in \eqref{l_*}.
This combined with \eqref{l_*} implies that
\eqref{estil} holds.

(i) We define
$$\mathcal F(\alpha):=\int_\Omega \frac{ \tilde c(\alpha) e^{\alpha m(x)} }{1+\tilde c(\alpha) e^{\alpha m(x)} }dx= |\Omega|- \int_\Omega \frac{1 }{1+\tilde c(\alpha) e^{\alpha m(x)} }dx,$$
 and consequently, $\ds\tilde l(\alpha)=\frac{r |\Omega|}{\mathcal F(\alpha)}$.
A direct computation yields
$$\mathcal F'(0)=\f{ |\Omega|}{ \left( |\Omega|+\int_\Omega m(x) dx  \right)^2}\left[ |\Omega| \int_\Omega m^2(x) dx-  \left(\int_\Omega m(x) dx \right)^2  \right].$$
Since $m(x)$ is non-constant, it follows from the H\"{o}lder inequality that
$\mathcal F'(0)>0$, and consequently $\tilde l'(0)<0$.

(ii) By \eqref{estil}, we see that $\ds\lim_{\alpha\to+\infty} \tilde l(\alpha)=\infty$ if $\ds\lim_{\alpha\to\infty} V(\alpha)=0$. Next, we give a sufficient condition for $\ds\lim_{\alpha\to\infty} V(\alpha)=0$.
We only consider the case that $m'(x)>0$ for all $x\in \Omega$, and the other case can be proved similarly.
Since
$$\int_{0}^{1}e^{\alpha m(x)}dx=\int_{0}^{1}e^{\alpha m(x)}m'(x)\f{1}{m'(x)}dx=\int_{0}^{1}e^{\alpha m(x)}\f{1}{m'(x)}dm(x),$$
which implies that
$$\f{e^{\alpha m(1)}-e^{\alpha m(0)}}{\alpha \ds\max_{x\in[0,1]}  m'(x)}\leq\int_{0}^{1}e^{\alpha m(x)}dx\leq \f{e^{\alpha m(1)}-e^{\alpha m(0)}}{\alpha \ds\min_{x\in[0,1]}  m'(x)}.$$
Therefore,
$$ V(\alpha)\le \frac{2\|m(x)\|_\infty ( \ds\max_{x\in[0,1]} m'(x))\left(e^{\alpha m(1)}-e^{\alpha m(0)}\right)^2}{\alpha( \ds\min_{x\in[0,1]} m'(x))^2\left(e^{2\alpha m(1)}-e^{2\alpha m(0)}\right)},$$
which yields $\ds\lim_{\alpha\to+\infty} V(\alpha)=0$.
\end{proof}
It follows from Proposition \ref{cl} that $\tilde l(\alpha)$ is strictly monotone decreasing when $\alpha$ is small, and it may change its monotonicity at least once under certain condition. We conjecture that $\tilde l(\alpha)$ changes its monotonicity even for general function $m(x)$ and all $\la>0$.

Now we solve \eqref{F} for $\la>0$ by virtue of the implicit function theorem.
\begin{lemma}\label{ll2}
For any
$ l_*>\tilde l$, where $\tilde l$ is defined in \eqref{l_*}, there exists $\tilde \delta_{l_*}>0$, a neighborhood $\mathcal O_{l_*}$ of
$(c_{0l_*},q_{0l_*},0,0)$ in $\mathbb R^2\times X_1^2$, and a continuously differentiable mapping
$$(\la,l)\mapsto\left(c^{(\la,l)},q^{(\la,l)},\xi^{(\la,l)},\eta^{(\la,l)}\right):
[0,\tilde \delta_{l_*}]\times [l_*-\tilde \delta_{l_*},l_*+\tilde \delta_{l_*}]\to \mathbb R^2\times X_1^2$$ such that
$\left(c^{(\la,l)},q^{(\la,l)},\xi^{(\la,l)},\eta^{(\la,l)}\right)\in \mathbb R^2\times X_1^2$ is a unique solution of
\eqref{F} in $\mathcal O_{l_*}$ for $(\la,l)\in [0,\tilde \delta_{l_*}]\times [l_*-\tilde \delta_{l_*},l_*+\tilde \delta_{l_*}]$. Moreover,
\begin{equation*}
\left(c^{(\la,l)},q^{(\la,l)},\xi^{(\la,l)},\eta^{(\la,l)}\right)=(c_{0l_*},q_{0 l_*},0,0) \;\;\text{for}\;\; (\la,l)=(0,l_*).
\end{equation*}
Here $c_{0l}$ and $q_{0l}$ are defined in Lemma \ref{l1}.
\end{lemma}
\begin{proof}
It follows from Lemma \ref{l1} that
$$ F(c_{0l_*},q_{0l_*}, 0, 0,l_*,0)= 0,$$
where $ F$ is defined in \eqref{F}. Then
the Fr\'echet derivative of $ F$ with respect to $(c,q,\xi,\eta)$ at $(c_{0l_*},q_{0l_*}, 0, 0,l_*,0)$ is as follows:
\begin{equation*}
 G(\hat c,\hat q,\hat \xi,\hat \eta)=(g_1,g_2,g_3,g_4)^T,
\end{equation*}
where $\hat c,\hat q\in\mathbb R$, $\hat \xi,\hat \eta \in X_1$, and
\begin{equation*}
\begin{cases}
 g_1(\hat c,\hat q,\hat \xi,\hat \eta):=&\ds\int_\Omega  e^{\alpha m(x)} \left(m(x)-2 c_{0l_*} e^{\alpha m(x)}  - \frac{ q_{0l_*}}{(1+c_{0l_*} e^{\alpha m(x)} )^2}\right)  (\hat c+ \hat \xi) dx\\
 & \ds - \int_\Omega\frac{c_{0l_*} e^{\alpha m(x)} }{1+ c_{0l_*} e^{\alpha m(x)} }(\hat q+ \hat \eta)  dx,\\
 g_2(\hat c,\hat q,\hat \xi,\hat \eta):=&L\hat \xi ,\\
 g_3(\hat c,\hat q,\hat \xi,\hat \eta):=& \ds \int_\Omega \frac{ l_* q_{0l_*} e^{\alpha m(x)} }{(1+c_{0l_*} e^{\alpha m(x)} )^2}\left(\hat c+ \hat \xi\right) dx\\
 &+\ds\int_\Omega\left(-r+\ds\f{l_*c_{0l_*}e^{\alpha m(x)}}{1+ c_{0l_*} e^{\alpha m(x)}}\right)(\hat q+\hat \eta)dx,\\
 g_4(\hat c,\hat q,\hat \xi,\hat \eta):=&\theta \Delta\hat \eta.\\
\end{cases}
\end{equation*}

If $ G(\hat c,\hat q,\hat \xi,\hat \eta)= 0$, then $\tilde\xi= 0$ and $ \hat \eta= 0$.
Substituting $\hat \xi=\hat \eta= 0$ into $g_1=0$ and $g_3=0$, respectively, we have
\begin{equation*}
(P_{ij})(\hat c,\hat q)^T=(0,0)^T,
\end{equation*}
where
\begin{equation*}
\begin{split}
P_{11}=&\ds\int_\Omega  { e^{\alpha m(x)} m(x) } dx -2 c_{0l_*} \int_\Omega e^{2\alpha m(x)} dx  -q_{0l_*} \int_\Omega \frac{  e^{\alpha m(x)}  }{(1+ c_{0l_*} e^{\alpha m(x)} )^2} dx,\\
P_{12}=&-\int_\Omega \frac{  c_{0l_*} e^{\alpha m(x)}   }{1+ c_{0l_*} e^{\alpha m(x)} } dx=-\ds\frac{r  |\Omega|}{l_*},\\
P_{21}=&\ds \int_\Omega \frac{ l_* q_{0l_*} e^{\alpha m(x)}  }{(1+ c_{0l_*} e^{\alpha m(x)} )^2}dx,\\
P_{22}=&\ds\int_\Omega\left(-r+\ds\f{l_*c_{0l_*}e^{\alpha m(x)}}{1+ c_{0l_*} e^{\alpha m(x)}}\right)dx=0.
\end{split}
\end{equation*}
Noticing that
\begin{equation*}
\text{det} (P_{ij})
= r  |\Omega| \int_\Omega \frac{ q_{0l_*} e^{\alpha m(x)}  }{(1+  c_{0l_*} e^{\alpha m(x)})^2}dx \ne 0,
\end{equation*}
we obtain that $\hat c=0$ and $\hat q=0$. Therefore, $ G$ is injective and thus bijective. Then, we can complete the proof by the implicit function theorem.
\end{proof}

By virtue of Lemma \ref{ll2}, we have the following result.
\begin{theorem}\label{thlocal}
Assume that $l_*>\tilde l$, where $\tilde l$ is defined in \eqref{l_*}.
Let \begin{equation}\label{uvf}
u^{(\la,l)}=c^{(\la,l)}+\xi^{(\la,l)},\;\;v^{(\la,l)}=q^{(\la,l)}+\eta^{(\la,l)}\;\;\text{for}\;\;(\la,l)\in [0,\delta_{l_*}]\times [l_*-\delta_{l_*},l_*+\delta_{l_*}],
\end{equation}
where $0<\delta_{l_*}\ll1$, and
$\left(c^{(\la,l)},q^{(\la,l)},\xi^{(\la,l)},\eta^{(\la,l)}\right)$ is obtained in Lemma \ref{ll2}.
Then $\left(u^{(\la,l)},v^{(\la,l)}\right)$ is the unique positive solution  of
 \eqref{equi1} for $(\la,l)\in (0,\delta_{l_*}]\times [l_*-\delta_{l_*},l_*+\delta_{l_*}]$.
Moreover,
\begin{equation}\label{limcruv}
\left(c^{(0,l)},q^{(0,l)},\xi^{(0,l)},\eta^{(0,l)}\right)=(c_{0l},q_{0l}, 0, 0)\;\;\text{for}\;\;
l\in[l_*-\delta_{l_*},l_*+\delta_{l_*}],
\end{equation}
where $c_{0l}$ and $q_{0l}$ are defined in Lemma \ref{l1}.
\end{theorem}
\begin{proof}
It follows from Lemma \ref{ll2} that when $\delta_{l_*}<\tilde\delta_{l_*}$,
$\left(u^{(\la,l)},v^{(\la,l)}\right)$ is a solution of
 \eqref{equi1} for $(\la,l)\in (0,\delta_{l_*}]\times [l_*-\delta_{l_*},l_*+\delta_{l_*}]$, and
 \begin{equation*}
\lim_{(\la,l)\to(0,l_*)}\left(u^{(\la,l)},v^{(\la,l)}\right)=\left(c_{0l_*},q_{0l_*}\right)\;\; \text{in}\;\;X^2.
\end{equation*}
Note from Lemma \ref{l1} that $c_{0l_*},q_{0l_*}>0$ if $l_*>\tilde l$. Then $\left(u^{(\la,l)},v^{(\la,l)}\right)$ is a positive solution of \eqref{equi1} for  $(\la,l)\in (0,\delta_{l_*}]\times [l_*-\delta_{l_*},l_*+\delta_{l_*}]$ with $0<\delta_{l_*}\ll1$.

Next, we show that $\left(u^{(\la,l)},v^{(\la,l)}\right)$ is the unique positive solution of \eqref{equi1} for  $(\la,l)\in (0,\delta_{l_*}]\times [l_*-\delta_{l_*},l_*+\delta_{l_*}]$ with $0<\delta_{l_*}\ll1$.
If it is not true, then there exists a sequence $\{(\la_k,l_k)\}_{k=1}^\infty$ such that
\begin{equation}\label{lasc}
0<\la_k\ll1,\;\; |l_k-l_*|\ll 1,\;\; \lim_{k\to\infty}(\la_k,l_k)=(0,l_*),
\end{equation}
and \eqref{equi1} admits a positive
solution $(u_k,v_k)$ for $(\la,l)=(\la_k,l_k)$ with
\begin{equation}\label{usa}
(u_k,v_k)\ne \left(u^{(\la_k,l_k)},v^{(\la_k,l_k)}\right).
\end{equation}
It follows from \eqref{decom} that $(u_k,v_k)$ can also be decomposed as follows:
\begin{equation*}
  u_k= c_k+ \xi_k,\;\; v_k= q_k+ \eta_k,\;\;\text{where}\;\;c_k,q_k\in\mathbb{R},\; \xi_k,\eta_k\in X_1.
\end{equation*}
Plugging $(u,v,\la,l)=(u_k, v_k,\la_k,l_k)$ into \eqref{equi1}, we see that
\begin{equation}\label{fisupp}
f_i\left(c_k,q_k,\xi_k,\eta_k,l_k,\la_k\right)=0\;\;\text{for}\;\;i=1,2,3,4,
\end{equation}
where $f_i\;(i=1,2,3,4)$ are defined in \eqref{fi}.
It follows from \eqref{equi1} that
\begin{equation}\label{lsupp}
\begin{split}
&- Lu_k  \leq \la_k e^{\alpha m(x)}  u_k\left(  m(x)  -  u_k\right),\\
&l_k \int_\Omega e^{\alpha m(x)}  u_k\left( m(x) - e^{\alpha m(x)}  u_k \right) dx=r \int_\Omega  v_k dx,
\end{split}
\end{equation}
where we have used the divergence formula to obtain the second equation.
From the maximum principle and the first equation of \eqref{lsupp}, we have
\begin{equation}\label{maxu}
0\leq u_k\leq \max_{x\in\overline\Omega} m(x) \;\;\text{for}\;\;k\ge1.
\end{equation}
This, together with the second equation of \eqref{lsupp}, implies that
$$
\int_\Omega v_k dx \leq \mathcal P_1:=\frac{\max_{k\ge1}l_k}{r}\max_{x\in\overline\Omega} m(x) \int_\Omega e^{\alpha m(x)} m(x) dx\;\;\text{for}\;\;k\ge1.
$$
Here we remark that $\max_{k\ge1}l_k<\infty$ from \eqref{lasc}.
Consequently,
\begin{equation}\label{M2}
\inf_{x\in\overline\Omega}  v_k\leq \mathcal P_2:=\frac{\mathcal P_1}{| \Omega | }\;\;\text{for}\;\;k\ge1.
\end{equation}
Note from \eqref{equi1} and \eqref{lasc} that
$$
-\theta\Delta   v_k+  r   v_k\ge-\theta\Delta   v_k+\la_k r v_k\geq 0.
$$
where we have used $0<\la_k\ll1$ (see \eqref{lasc}) for the first inequality. Note that $\Omega$ is a bounded domain in $\mathbb R^N$ with $1\le N\le 3$.
Then we see from \cite[Lemma 2.1]{peng2009} that there exists a positive constant $C_0$, depending only on $r$ and $\Omega$, such that
\begin{equation}\label{maxv}
\left\| v_k \right\|_2 \leq  C_0 \inf_{x\in\overline\Omega} v_k=C_0\mathcal P_2\;\;\text{for}\;\;k\ge1,
\end{equation}
where we have used \eqref{M2} in the last step.

Note from \eqref{maxu} and \eqref{maxv} that $ \{u_k\}_{k=1}^\infty $ is bounded in $L^{\infty} (\Omega)$, and $ \{v_k\}_{k=1}^\infty$ is bounded in $L^2 (\Omega)$. Since $\lim_{k\to\infty}\la_k=0$, we see from \eqref{fisupp} with $i=2,4$ that
\begin{equation*}
\lim_{k \to \infty} L\xi_k=0\;\;\text{and}\;\;\lim_{k \to \infty} \Delta \eta_k=0 \;\;\text{in}\;\;Y_1,
\end{equation*}
which implies that
\begin{equation}\label{limxieta}
\lim_{k \to \infty} \xi_k=0\;\;\text{and}\;\;\lim_{k \to \infty}  \eta_k=0 \;\;\text{in}\;\;X_1.
\end{equation}
Here $X_1$ and $Y_1$ are defined in \eqref{X1Y1}.
By \eqref{uv}, \eqref{maxu} and \eqref{maxv}, we see that
$$c_k=\frac{1}{| \Omega | } \int_\Omega u_k dx\;\; \text{and}\;\;  q_k=\frac{1}{| \Omega | } \int_\Omega  v_k dx,$$
and $\{c_k\}_{k=1}^\infty$,  $\{q_k\}_{k=1}^\infty$ are bounded.
Then, up to a subsequence, we see that
\begin{equation*}
\lim_{k \to \infty}  c_k=  c^*,\;\;\lim_{k \to \infty}  q_k=  q^*.
 \end{equation*}
Taking the limits of \eqref{fisupp} with $i=1,3$  on both sides as $k\to\infty$, respectively,
we obtain that $(c^*,q^*)$ satisfies \eqref{solcq} with $l=l_*$.

We first claim that
\begin{equation*}
(c^*,q^*)\neq(0,0).
\end{equation*}
Suppose that it is not true. Then by \eqref{limxieta} and the embedding theorems, we see that, up to a subsequence,
\begin{equation*}
\lim_{k\to\infty}(u_k,v_k)=(0,0) \;\;\text{in}\;\;C^{\gamma}(\overline\Omega)\;\;\text{for some}\;\;0<\gamma<1.
\end{equation*}
This yields, for sufficiently large $k$,
$$ \int_\Omega  v_k\left(-r+\frac{l_k e^{\alpha m(x)}  u_k}{1+  e^{\alpha m(x)} u_k}  \right)dx<0.$$
Substituting $(u,v,\la,l)=(u_k,v_k,\la_k,l_k)$ into \eqref{equi1}, and integrating the result over $\Omega$, we obtain that
$$ \int_\Omega  v_k\left(-r+\frac{l_k e^{\alpha m(x)}  u_k}{1+  e^{\alpha m(x)} u_k}  \right)dx=0,$$
which is a contradiction. Next, we show that
\begin{equation*}
(c^*,q^*)\neq(\tilde c,0).
\end{equation*}
By way of contradiction, $(c^*,q^*)=(\tilde c,0)$. Similarly, by \eqref{limxieta} and the embedding theorems,
 we see that, up to a subsequence,
\begin{equation*}
\lim_{k\to\infty}(u_k,v_k)=(\tilde c,0) \;\;\text{in}\;\;C^{\gamma}(\overline\Omega)\;\;\text{for some}\;\;0<\gamma<1.
\end{equation*}
From the second equation of \eqref{equi1}, we have
\begin{equation}\label{inu}
-r|\Omega|+\int_\Omega \frac{l_k e^{\alpha m(x)}u_k}{1+e^{\alpha m(x)}  u_k}dx=-\ds\f{\theta}{\la_k}\int_{\Omega}\ds\f{|\nabla v_k|^2}{v_k^2}dx\leq0.
\end{equation}
Then taking $k\to\infty$ on both sides of \eqref{inu} yields
\begin{equation*}
l_*\leq \frac{r |\Omega|}{\int_\Omega \frac{ \tilde c e^{\alpha m(x)} }{1+\tilde c e^{\alpha m(x)} }dx} = \tilde l,
\end{equation*}
 which is a contradiction.

Therefore,
\begin{equation*}
\lim _{k\to\infty}(c_k,q_k)= \left(c_{0l_*},q_{0l_*}\right).
 \end{equation*}
This, combined with \eqref{limxieta} and Lemma \ref{ll2}, implies that, for sufficiently large $k$,
\begin{equation*}
(c_k,q_k,\xi_k,\eta_k)=\left(c^{(\la_k,l_k)},\xi^{(\la_k,l_k)},q^{(\la_k,l_k)},\eta^{(\la_k,l_k)}\right),
\end{equation*}
and consequently, for sufficiently large $k$,
\begin{equation*}
(u_k,v_k)=\left(u^{(\la_k,l_k)},v^{(\la_k,l_k)}\right).
\end{equation*}
This contradicts \eqref{usa}, and the uniqueness is obtained.

By using similar arguments as in the proof of the uniqueness, we can show that, for any $l\in[l_*-\delta_{l_*},l_*+\delta_{l_*}]$,
\begin{equation*}
\lim_{\la\to0}\left(c^{(\la,l)},q^{(\la,l)},\xi^{(\la,l)},\eta^{(\la,l)}\right)=(c_{0l},q_{0l}, 0, 0)\;\;\text{in}\;\;\mathbb R^2\times X_1^2.
\end{equation*}
Note from Lemma \ref{ll2} that $\left(c^{(\la,l)},q^{(\la,l)},\xi^{(\la,l)},\eta^{(\la,l)}\right)$ is continuously differentiable for $(\la,l)\in [0,\delta_{l_*}]\times [l_*-\delta_{l_*},l_*+\delta_{l_*}]$. Then we see that
\eqref{limcruv} holds. This completes the proof.
\end{proof}
From Theorem \ref{thlocal}, we see that \eqref{equi1} admits a unique positive solution when $(l,\la)$ is in a small neighborhood of $(l_*,0)$ with $l_*>\tilde l$. In the following, we will solve \eqref{equi1} for a wide range of parameters, see the rectangular region in Fig. \ref{fig11}.
\begin{theorem}\label{thglobal}
Let $\mathcal L:=[ \tilde l+\epsilon,1/\epsilon]$, where $0<\epsilon\ll 1$ and $ \tilde l$ is defined in Lemma \ref{l1}. Then
the following statements hold.
\begin{enumerate}
\item [$\rm{(i)}$] There exists $\delta_{\epsilon} >0$ such that, for $(\la,l)\in (0,\delta_\epsilon]\times \mathcal L$,
model \eqref{equi1} admits a unique positive solution
$(u^{(\la,l)},v^{(\la,l)})$.
\item [$\rm{(ii)}$] Let $(u^{(0,l)},v^{(0,l)})=\left(c_{0l},q_{0l}\right)$ for $l\in\mathcal L$,
where $\left(c_{0l},q_{0l}\right)$ is defined in Lemma \ref{l1}. Then $(u^{(\la,l)},v^{(\la,l)})$ is continuously differentiable for $(\la,l)\in [0,\delta_\epsilon]\times \mathcal L$, and $(u^{(\la,l)},v^{(\la,l)})$ can be decomposed as follows:
\begin{equation*}
u^{(\la,l)}=c^{(\la,l)}+\xi^{(\la,l)},\;\;v^{(\la,l)}=q^{(\la,l)}+\eta^{(\la,l)},
\end{equation*}
where
$\left(c^{(\la,l)},q^{(\la,l)},\xi^{(\la,l)},\eta^{(\la,l)}\right)\in \mathbb R^2\times X_1^2$ solves Eq. \eqref{F} for $(\la,l)\in [0,\delta_\epsilon]\times \mathcal L$.
\end{enumerate}
\end{theorem}
\begin{proof}
It follows from Lemma \ref{thlocal} that, for any $l_*\in\mathcal L$, there exists $\delta_{l_*}$ ($0<\delta_{l_*}\ll1$) such that, for $(\la,l)\in (0,\delta_{l_*}]\times [ l_*-\delta_{l_*}, l_*+\delta_{l_*}]$,
system \eqref{equi1} admits a unique positive solution
$(u^{(\la,l)},v^{(\la,l)})$,
where $u^{(\la,l)}$ and $v^{(\la,l)}$ are defined in \eqref{uvf} and continuously differentiable for $(\la,l)\in [0,\delta_{l_*}]\times [l_*-\delta_{l_*},l_*+\delta_{l_*}]$.
Clearly,
$$\mathcal L \subseteq \bigcup_{ l_*\in\mathcal L}  {\left(  l_* - \delta_{l_*}, l_* +\delta_{l_*} \right)}. $$
Noticing that $\mathcal L$ is compact, we see that there exist finite open intervals, denoted by $\left( {l_*^{(i)}} -\delta_{ l^{(i)}_*},l^{(i)}_* +\delta_{l^{(i)}_*} \right)$ for $i=1,\ldots,s$, such that
$$\mathcal L \subseteq \bigcup_{i=1}^s  \left( {l_*^{(i)}} -\delta_{ l^{(i)}_*},{l_*^{(i)}} +\delta_{l^{(i)}_*} \right).$$
Choose $ \delta_\epsilon= \min_{1 \le i\le s} \delta _{ l_*^{(i)}}$. Then, for $(\la,l)\in (0,\delta_\epsilon]\times \mathcal L$,
 system \eqref{equi1} admits a unique positive solution $(u^{(\la,l)},v^{(\la,l)})$. By Lemma \ref{ll2} and Theorem \ref{thlocal}, we see that $(u^{(\la,l)},v^{(\la,l)})$ is continuously differentiable on $[0,\delta_\epsilon]\times \mathcal L$, if $(u^{(0,l)},v^{(0,l)})=\left(c_{0l},q_{0l}\right)$.
\end{proof}
Using similar arguments as in Lemma \ref{ll2} and Theorems \ref{thlocal} and \ref{thglobal}, we can show the nonexistence of positive solution under certain condition, and here we omit the proof for simplicity.
\begin{theorem}\label{thglobal2}
Let $\mathcal L_1:=[\epsilon_1,\tilde l-\epsilon_1]$, where $0<\epsilon_1\ll 1$ and $ \tilde l$ is defined in \eqref{l_*}. Then there exists $\delta_{\epsilon_1} >0$ such that, for $(\la,l)\in (0,\delta_{\epsilon_1}]\times \mathcal L_1$,
model \eqref{equi1} admits no positive solution.
\end{theorem}
\section{Stability and Hopf bifurcation}

Throughout this section, we let $\left(u^{(\la,l)},v^{(\la,l)}\right)$ be the unique positive solution of model \eqref{equi1} obtained in Theorem \ref{thglobal}.
In the following, we use $l$ as the bifurcation parameter, and show that under certain condition
there exists a Hopf bifurcation curve $l=l_\la$ when $\la$ is small.

Linearizing model \eqref{m4} at $\left(u^{(\la ,l)}, v^{(\la ,l)}\right)$, we obtain
\begin{equation*}
\begin{cases}
\tilde u_t= e^{-\alpha m(x)} \nabla  \cdot \left[e^{\alpha m(x)} \nabla \tilde u \right]  +\la \left( M_{1}^{(\la,l)}\tilde u +M_{2}^{(\la,l)} \tilde v\right),&x \in  \Omega,\;t > 0,\\
 \tilde v_t = \theta \Delta \tilde v  + \la \left( M_{3}^{(\la,l)}\tilde u +M_{4} ^{(\la,l)}\tilde v\right), &x \in  \Omega,\;t > 0,\\
\partial_n \tilde u = \partial_n \tilde v = 0, & x \in \partial \Omega,\;t > 0,\\
\end{cases}
\end{equation*}
where
\begin{equation}\label{Mi}
\begin{split}
&\ds M_{1}^{(\la,l)}= m(x)-2 e^{\alpha m(x)} u^{(\la ,l)}-\frac{v^{(\la ,l)}}{\left(1+  e^{\alpha m(x)} u^{(\la ,l)}  \right)^2} ,\;\;M_{2}^{(\la,l)}=-\frac{u^{(\la ,l)}}{1+e^{\alpha m(x)} u^{(\la ,l)}},\\
&\ds M_{3}^{(\la,l)}=\frac{l  e^{\alpha m(x)} v^{(\la ,l)}}{\left(1+ e^{\alpha m(x)} u^{(\la ,l)}  \right)^2},\;\;M_{4}^{(\la,l)}=-r+\frac{l e^{\alpha m(x)} u^{(\la ,l)}}{1+ e^{\alpha m(x)} u^{(\la ,l)} }.
\end{split}
\end{equation}
Let\begin{equation}\label{Al}
 A_l(\la) : = \left( {\begin{array}{cc}
e^{-\alpha m(x)} \nabla  \cdot \left[e^{\alpha m(x)} \nabla \right]&0\\
0&\theta \Delta
\end{array}} \right) + \la\left( {\begin{array}{cc}
M_1^{(\la,l)}&M_2^{(\la,l)}\\
M_3^{(\la,l)}& M_4^{(\la,l)}
\end{array}} \right).
\end{equation}
Then, $\mu \in \mathbb C$ is an eigenvalue of $A_{l}(\la)$ if and only if there exists $ (\varphi, \psi)^T\left(\ne (0,0)^T\right)\in X^2_{\mathbb C}$ such that
\begin{equation}\label{eig1}
\begin{cases}
&\ds e^{-\alpha m(x)} \nabla  \cdot \left[e^{\alpha m(x)} \nabla \varphi \right]  +\la \left( M_{1}^{(\la,l)} \varphi + M_{2}^{(\la,l)} \psi \right)-\mu \varphi =0, \\
&\ds \theta\Delta \psi  + \la \left( M_{3}^{(\la,l)} \varphi +M_{4}^{(\la,l)} \psi \right)-\mu \psi  =0.
\end{cases}
\end{equation}

We first give a \textit{priori} estimates for solutions of \eqref{eig1} for later use.

\begin{lemma} \label{l3}
Let $\mathcal L$ and $\delta_\epsilon$ be defined and obtained in Theorem \ref{thglobal}.
Suppose that $(\mu_\la ,l_\la, \varphi_\la, \psi_\la)$ solves \eqref{eig1} for $\la\in(0,\delta_\epsilon]$, where $\mathcal Re \mu_\la\ge0$, $(\varphi_\la, \psi_\la)^T\left(\ne (0,0)^T\right)\in X^2_{\mathbb C}$ and $l_\la\in \mathcal L$. Then $\left|\mu_\la/\la\right|$ is bounded for $\la \in (0, \delta_\epsilon]$.
\end{lemma}

\begin{proof}
Substituting $(\mu_\la ,l_\la, \varphi_\la, \psi_\la)$ into Eq. \eqref{eig1}, we have
\begin{equation}\label{eig12}
\begin{cases}
&\ds e^{-\alpha m(x)} \nabla  \cdot \left[e^{\alpha m(x)} \nabla \varphi_\la \right]  +\la \left( M_{1}^{(\la,l_\la)} \varphi_\la + M_{2}^{(\la,l_\la)} \psi_\la \right)-\mu_\la \varphi_\la =0, \\
&\ds \theta\Delta \psi_\la  + \la \left( M_{3}^{(\la,l_\la)} \varphi_\la +M_{4}^{(\la,l_\la)} \psi_\la \right)-\mu_\la \psi_\la  =0.
\end{cases}
\end{equation}
Then multiplying the first and second equations of \eqref{eig12} by $  e^{\alpha m(x)}\overline \varphi_\la$ and $\overline \psi_\la$, respectively,
summing these two equations, and integrating the result over $\Omega$, we have
\begin{equation}\label{infesti}
\begin{split}
\mu_\la \int_\Omega \left(e^{\alpha m(x)}  |\varphi _{\la}|^2+|\psi _{\la}|^2 \right)dx=&\langle \varphi _{\la},  \nabla  \cdot \left[e^{\alpha m(x)} \nabla \varphi_\la \right] \rangle+\langle \psi  _{\la},  \Delta \psi_\la  \rangle\\
&+\la \int_\Omega e^{\alpha m(x)} \left( M_{1}^{(\la,l_\la)} |\varphi _{\la}|^2+M_{2}^{(\la,l_\la)}\overline\varphi _{\la} \psi _{\la}\right)dx\\
&+\la \int_\Omega\left( M_{3}^{(\la,l_\la)}\varphi _{\la} \overline \psi _{\la} +M_{4}^{(\la,l_\la)} |\psi _{\la}|^2 \right)dx.
\end{split}
\end{equation}
We see from the divergence theorem that
\begin{equation}\label{divg}
\langle \varphi _{\la},  \nabla  \cdot \left[e^{\alpha m(x)} \nabla \varphi_\la \right] \rangle+\langle   \psi  _{\la},  \Delta \psi_\la  \rangle=-\int_\Omega e^{\alpha m(x)} |\nabla\varphi _{\la} |^2dx-\int_\Omega  |\nabla\psi _{\la} |^2dx\leq0.\\
\end{equation}
Noticing from Theorem \ref{thglobal} that
$\left(u^{(\la,l)},v^{(\la,l)}\right):[0,\delta_\epsilon]\times \mathcal L\to X^2 $ is continuously differentiable,
we see from the imbedding theorems that there exists a positive constant $\mathcal P_*$ such that
\begin{equation}\label{bondM}
\left\|M^{(\la,l)}_i\right\|_\infty\le \mathcal P_*\;\;\text{for}\;\;(\la,l)\in [0,\delta_\epsilon]\times \mathcal L\;\;\text{and}\;\;i=1,2,3,4.
\end{equation}
Then, we see from \eqref{infesti}-\eqref{bondM} that
\begin{equation*}
\begin{split}
0\le \mathcal Re\left( \frac{\mu_\la}{\la}\right)\leq& \mathcal P_*\left(1+\ds\f{\int_\Omega \left( e^{\alpha m(x)} \overline\varphi _{\la} \psi _{\la} + \varphi _{\la} \overline \psi _{\la} \right)dx}{ \int_\Omega \left(e^{\alpha m(x)}  |\varphi _{\la}|^2+|\psi _{\la}|^2 \right)dx}\right)\\
\leq&\mathcal P_*\left(1+e^{\alpha \max_{x\in\overline\Omega}m(x)}\right).
\end{split}
\end{equation*}
Similarly, we can obtain that
\begin{equation*}
\left|\mathcal Im\left( \frac{\mu_\la}{\la}\right)\right| \leq \mathcal P_*\left(1+e^{\alpha \max_{x\in\overline\Omega}m(x)}\right).
\end{equation*}
This completes the proof.
\end{proof}

To analyze the stability of
$\left(u^{(\la,l)},v^{(\la,l)}\right)$, we need to consider whether the eigenvalues of \eqref{eig1} could pass through the imaginary axis.
It follows from Lemma \ref{l3} that if $\mu={\rm i} \sigma$ is an eigenvalue of \eqref{eig1}, then $\nu=\sigma/\la$ is bounded for $\la\in(0,\delta_\epsilon]$, where $\delta_\epsilon$ is obtained in Theorem \ref{thglobal}.
Substituting $\mu={\rm i}\la \nu\;(\nu\geq0)$ into \eqref{eig1}, we have
\begin{equation}\label{v1}
\begin{cases}
&\ds L \varphi  +\la e^{\alpha m(x)} \left( M_{1}^{(\la,l)} \varphi + M_{2}^{(\la,l)} \psi \right)-{\rm i}\la \nu e^{\alpha m(x)} \varphi =0, \\
&\ds \theta\Delta \psi  + \la \left( M_{3}^{(\la,l)} \varphi +M_{4}^{(\la,l)} \psi \right)-{\rm i}\la \nu \psi  =0.
\end{cases}
\end{equation}
Ignoring a scalar factor,
$( \varphi, \psi)^T\in X^2_{\mathbb C}$ in \eqref{v1} can be decomposed as follows:
\begin{equation}\label{vps}
\begin{cases}
 \varphi =\delta+ w, \;\;\text{where}\; \delta\ge0\;\;\text{and}\;\; w\in \left(X_{1}\right)_{\mathbb C},\\
 \psi  =(s_{1}+{\rm i}s_{ 2})+ z,\;\;\text{where}\;s_1,s_2\in\mathbb R\;\;\text{and}\;\; z \in \left(X_{1}\right)_{\mathbb C},\\
\| \varphi\|_2^{2}+\| \psi\|_2^{2}=| \Omega |.
\end{cases}
\end{equation}

Now we can obtain an equivalent problem of \eqref{v1} in the following.
\begin{lemma}\label{lema32}
Let  $(  \varphi,  \psi)$ be defined in \eqref{vps}. Then
$(\varphi,  \psi,\nu,l)$ solves \eqref{v1} with $\nu\ge0$, $l\in\mathcal L$, if and only if $( \delta, s_1,s_2, \nu,w,  z,l)$ solves
\begin{equation}\label{H}
\begin{cases}
  { H}(\delta, s_1,s_2, \nu, w,z,l,\la)=   0,\\
\delta\ge0,\;s_1 ,s_2\in\mathbb R,\;\nu\geq0,\;l\in\mathcal L,\; w,  z\in \left(X_{1}\right)_{\mathbb C}.
 \end{cases}
\end{equation}
Here
$$ H(\delta, s_1,s_2, \nu, w,z,l,\la)=( h_1, h_2,h_3,h_4,h_5)^T$$ is a continuously differentiable mapping from $\mathbb R^4\times \left(\left(X_{1}\right)_{\mathbb C}\right)^2\times\mathcal L\times [0,\delta _\epsilon]$ to $\left(\mathbb C\times \left(Y_{1}\right)_{\mathbb C}\right)^2\times \mathbb R$,
 where
\begin{equation*}
\begin{split}
h_1(\delta, s_1,s_2, \nu, w,z,l,\la):=&\ds \int_\Omega  {  e^{\alpha m(x)} \left[ M_{1}^{(\la,l)}(\delta+w)+ M_{2}^{(\la,l)} (s_1 + {\rm i} s_2+z)  \right]} dx\\
&-{\rm i} \nu  \int_\Omega e^{\alpha m(x)} (\delta+w)dx ,\\
h_{2}(\delta, s_1,s_2, \nu, w,z,l,\la):=&\ds  L w  +\la e^{\alpha m(x)}  \left[ M_{1}^{(\la,l)} (\delta+w) + M_{2}^{(\la,l)} (s_1 + {\rm i} s_2+z)  \right]\\
&-{\rm i}\la \nu e^{\alpha m(x)}  (\delta+w) -\frac{\la }{ | \Omega |} h_1,\\
h_{3}(\delta, s_1,s_2, \nu, w,z,l,\la):=& \ds \int_\Omega  {  \left[ M_{3}^{(\la,l)}(\delta+w)+ M_{4}^{(\la,l)} (s_1 + {\rm i} s_2+z)  \right]} dx\\
&-{\rm i} \nu (s_1 + {\rm i} s_2)| \Omega | ,\\
h_{4}(\delta, s_1,s_2, \nu, w,z,l,\la):=&\ds \theta\Delta z  + \la \left[ M_{3}^{(\la,l)} (\delta+w) +M_{4}^{(\la,l)} (s_1 + {\rm i} s_2+z) \right]\\
&-{\rm i}\la \nu (s_1 + {\rm i} s_2+z)  -\frac{\la }{ | \Omega |} h_3,\\
h_{5}( \delta,s_1,s_2,  w,  z,\nu,l,\la):=&\ds  | \Omega |\left(\delta^2 +s_1^2+s_2^2-1\right) +\| w\|_2^2+\| z\|_2^2,
\end{split}
\end{equation*}
and $M_{i}^{(\la,l)}$ are defined in \eqref{Mi} for $i=1,2,3,4$.
\end{lemma}
\begin{proof}
Note that $f=c+z$ for any $f\in Y_{\mathbb C}$, where
\begin{equation}\label{fcom}
c=\f{1}{|\Omega|}\int_\Omega f dx\in\mathbb C, \;\;
z=f-\f{1}{|\Omega|}\int_\Omega fdx\in \left(Y_1\right)_{\mathbb C}.
\end{equation}
A direct computation implies that
${ H}(\delta, s_1,s_2, \nu, w,z,l,\la)$ is a continuously differentiable mapping from $\mathbb R^4\times \left(\left(X_{1}\right)_{\mathbb C}\right)^2\times\mathcal L\times [0,\delta _\epsilon]$ to $\left(\mathbb C\times \left(Y_{1}\right)_{\mathbb C}\right)^2\times \mathbb R$.
Denote the left sides of the two equations of \eqref{v1} as $G_1$ and $G_2$, respectively. That is,
\begin{equation*}
\begin{split}
&G_1(\varphi,  \psi,\nu,l,\la):=\ds L \varphi  +\la e^{\alpha m(x)} \left( M_{1}^{(\la,l)} \varphi + M_{2}^{(\la,l)} \psi \right)-{\rm i}\la \nu e^{\alpha m(x)} \varphi, \\
&G_2(\varphi,  \psi,\nu,l,\la):=\ds \theta\Delta \psi  + \la \left( M_{3}^{(\la,l)} \varphi +M_{4}^{(\la,l)} \psi \right)-{\rm i}\la \nu \psi.
\end{split}
\end{equation*}
Plugging \eqref{vps} into \eqref{v1}, wee see from \eqref{fcom} that
$G_1(\varphi,  \psi,\nu,l,\la)=0$ if and only if $$h_i(\delta, s_1,s_2, \nu, w,z,l,\la)=0\;\; \text{for}\;\; i=1,2,$$ and $G_2(\varphi,  \psi,\nu,l,\la)=0$ if and only if $$h_i(\delta, s_1,s_2, \nu, w,z,l,\la)=0\;\; \text{for}\;\; i=3,4.$$ This completes the proof.
\end{proof}
Note from \eqref{Mi} that
\begin{equation}\label{Mi0}
\begin{split}
&\ds M_{1}^{(0,l)} = m(x)-2 c _{0l} e^{\alpha m(x)} -\frac{ q _{0l} }{\left(1+  c _{0l} e^{\alpha m(x)} \right)^2} ,\;\; M_{2}^{(0,l)}=-\frac{ c _{0l}   }{1+ c _{0l} e^{\alpha m(x)} }<0,\\
&\ds M_{3}^{(0,l)} dx=\frac{l q _{0l} e^{\alpha m(x)}  }{\left(1+ c _{0l} e^{\alpha m(x)}  \right)^2} >0 ,\;\;M_{4}^{(0,l)} =-r+\frac{ l c _{0l} e^{\alpha m(x)} }{1+ c _{0l} e^{\alpha m(x)} },
\end{split}
\end{equation}
where $c _{0l}$ and $q_{0l}$ are defined in Lemma \ref{l1}. Then we give the following result for further application.
\begin{lemma}\label{suppst}
Let  $\mathcal S(l):=\int_\Omega e^{\alpha m(x)} M_1^{(0,l)}dx$, where $l\in\mathcal L$, and $\mathcal L:=[ \tilde l+\epsilon,1/\epsilon]$ is defined in Theorem \ref{thglobal} with $0<\epsilon\ll1$,
and let $\mathcal T (\alpha):=\int_\Omega e^{\alpha m(x)}( m(x) -1) dx$. Then the following two statements hold:
\begin{enumerate}
\item [$\rm{(i)}$] If $\mathcal T (\alpha)<0$, then $\mathcal S(l)<0$ for all $l\in \mathcal L$;
\item [$\rm{(ii)}$] If $\mathcal T (\alpha)>0$, then there exists $l_0\in {int}( \mathcal  L)=(\tilde l+\epsilon,1/\epsilon)$ such that $\mathcal S(l_0)=0$, $\mathcal S'(l_0)>0$, $\mathcal S(l)<0$ for $l\in[ \tilde l+\epsilon,l_0)$, and $\mathcal S(l)>0$ for $l\in(l_0,1/\epsilon]$.
\end{enumerate}
\end{lemma}
\begin{proof}
We construct an auxiliary function:
\begin{equation*}
\mathcal S_1(c)= \int_\Omega  \frac{ e^{\alpha m(x)} }{ 1+ c   e^{\alpha m(x)} } dx.
\end{equation*}
A direct computation implies that
\begin{equation}\label{deriv}
\mathcal S'_1(c )=-\int_\Omega  \frac{ e^{2\alpha m(x)} }{ \left(1+ c   e^{\alpha m(x)} \right)^2 } dx<0,\;\;\mathcal S''_1(c )=2\int_\Omega  \frac{ e^{3\alpha m(x)} }{ \left(1+ c   e^{\alpha m(x)} \right)^3 } dx>0.\\
\end{equation}
Clearly, we see from \eqref{Mi0} that
\begin{equation}\label{sl1}
\mathcal S(l)=\int_\Omega \left[e^{\alpha m(x)} m(x)-2 c _{0l} e^{2 \alpha m(x)}\right] dx- q _{0l}  \int_\Omega  \frac{ e^{\alpha m(x)} }{\left(1+ c _{0l}  e^{\alpha m(x)} \right)^2} dx.
\end{equation}
It follows from the first equation of \eqref{solcq2} that
\begin{equation}\label{q0l}
 q _{0l} =\ds\f{1}{\mathcal S_1(c_{0l})}\int_\Omega \left[e^{\alpha m(x)} m(x)- c _{0l} e^{2 \alpha m(x)}\right] dx,
\end{equation}
and plugging \eqref{q0l} into \eqref{sl1}, we get
\begin{equation}\label{ssov}
\mathcal S(l)=\mathcal S_2(c_{0l}).
\end{equation}
Here
\begin{equation*}
\begin{split}
\mathcal S_2(c):=&\int_\Omega \left[e^{\alpha m(x)} m(x)-2 c e^{2 \alpha m(x)}\right] dx\\
&-\ds\f{1}{\mathcal S_1(c)}\int_\Omega \left[e^{\alpha m(x)} m(x)- c e^{2 \alpha m(x)}\right] dx\ds\int_\Omega  \frac{ e^{\alpha m(x)} }{\left(1+ c  e^{\alpha m(x)} \right)^2} dx\ds\\
=&\ds\int_\Omega \left[e^{\alpha m(x)} m(x)- c e^{2 \alpha m(x)}\right] dx\left[1-\ds\f{1}{\mathcal S_1(c)}\ds\int_\Omega  \frac{ e^{\alpha m(x)} }{\left(1+ c  e^{\alpha m(x)} \right)^2} dx\right]-c\int_\Omega  e^{2 \alpha m(x)} dx\\
=&-\ds\f{c\mathcal S'_1(c)}{\mathcal S_1(c)}\ds\int_\Omega \left[e^{\alpha m(x)} m(x)- c e^{2 \alpha m(x)}\right] dx-c\int_\Omega  e^{2 \alpha m(x)} dx,\\
\end{split}
\end{equation*}
where we have used \eqref{deriv} in the last step.
Let
\begin{equation}\label{s3c}
\mathcal S_3(c)=\ds\f{\mathcal S'_1(c)}{\mathcal S_1(c)}\ds\int_\Omega \left[e^{\alpha m(x)} m(x)- c e^{2 \alpha m(x)}\right] dx+\int_\Omega  e^{2 \alpha m(x)} dx,
\end{equation}
and consequently, we have
\begin{equation}\label{sc23}
\mathcal S_2(c)=-c\mathcal S_3(c).
\end{equation}
It follows from the first equation of \eqref{c0l} that
\begin{equation}\label{cc0l}
\ds\f{dc_{0l}}{dl}<0,\;\;\lim_{l\to \tilde l}c_{0l}=\tilde c, \;\;\text{and}\;\; \lim_{l\to \infty}c_{0l}=0,
 \end{equation}
 where $\tilde c$ and $ \tilde l$ are defined in Lemma \ref{l1} (see \eqref{l_*}).
Therefore, to determine the zeros of $\mathcal S(l)$ in $(\tilde l,\infty)$, we only need to consider the zeros of $\mathcal S_3(c)$ in $(0,\tilde c)$.

 It follows from the H\"{o}lder inequality and \eqref{deriv} that
\begin{equation*}
\begin{split}
[\mathcal S_1'(c)]^2&=\left[\int_\Omega  \frac{ e^{\frac{3}{2}\alpha m(x)} e^{\frac{1}{2}\alpha m(x)} }{ \left(1+ c e^{\alpha m(x)} \right)^\frac{3}{2}  \left(1+ c  e^{\alpha m(x)} \right)^\frac{1}{2}} dx\right]^2\\
&\leq \int_\Omega  \frac{ e^{3\alpha m(x)} }{ \left(1+ c e^{\alpha m(x)} \right)^3 } dx\int_\Omega  \frac{ e^{\alpha m(x)} }{ 1+ c  e^{\alpha m(x)} } dx=\frac{\mathcal S_1''(c)\mathcal S_1(c)}{2},
\end{split}
\end{equation*}
and consequently,
\begin{equation}\label{f3g222}
\left[\ds\f{\mathcal S'_1(c)}{\mathcal S_1(c)}\right]'=\frac{\mathcal S''_1(c )\mathcal S_1(c )-\left(\mathcal S'_1(c )\right)^2}{\mathcal S^2_1(c )}>0.
\end{equation}
Note from \eqref{l_*} that $\int_\Omega \left[e^{\alpha m(x)} m(x)- c e^{2 \alpha m(x)}\right] dx>0$ for $c\in(0,\tilde c)$.
This, combined with \eqref{deriv}, \eqref{s3c} and \eqref{f3g222}, yields
\begin{equation}\label{mons3}
\mathcal S'_3(c )=\left[\ds\f{\mathcal S'_1(c)}{\mathcal S_1(c)}\right]'\int_\Omega \left[e^{\alpha m(x)} m(x)- c e^{2 \alpha m(x)}\right] dx-\ds\f{\mathcal S'_1(c)}{\mathcal S_1(c)}\int_\Omega e^{2 \alpha m(x)} dx>0.
\end{equation}
It follows from \eqref{l_*} that $\lim_{c\to\tilde c}\mathcal S_3(c)=\int_\Omega  e^{2 \alpha m(x)} dx>0$, and
\begin{equation*}
\lim_{c\to0} \mathcal S_3(c)=-\frac{\int_\Omega e^{2\alpha m(x)} dx}{\int_\Omega  e^{\alpha m(x)} dx}\mathcal T(\alpha).
\end{equation*}
Therefore, if $\mathcal T(\alpha)<0$, then $\mathcal S_3(c)>0$ for all $c\in(0,\tilde c)$.
If $\mathcal T(\alpha)>0$, then there exists $c_0\in(0,\tilde c)$ such that $\mathcal S_3(c_0)=0$, $\mathcal S_3(c)<0$ for $c\in(0,c_0)$ and $\mathcal S_3(c)>0$ for $c\in(c_0,\tilde c)$. Then, we see from \eqref{ssov}, \eqref{sc23}  and \eqref{cc0l} that
if $\mathcal T(\alpha)<0$, then (i) holds; and if
$\mathcal T(\alpha)>0$,
then there exists  $l_0>\tilde l$ such that
\begin{equation}\label{coadv}
c_{0l_0}=c_0,\;\;\mathcal S(l_0)=0,\;\;\mathcal S(l)<0\;\;\text{for}\;\;l\in(\tilde l,l_0),\;\;\text{and}\;\mathcal S(l)>0\;\;\text{for}\;\;l\in(l_0,\infty).
\end{equation}

It follows from \eqref{ssov} and \eqref{sc23} that
\begin{equation*}
\begin{split}
\mathcal S'(l_0)=&\left.\mathcal S'_2\left(c_{0l}\right)\right|_{l=l_0}\left.\ds\f{dc_{0l}}{dl}\right|_{l=l_0}=-\left[\mathcal S_3\left(c_{0l_0}\right)+c_{0l_0}\left. \mathcal S_3'(c_{0l})\right|_{l=l_0}\right]\left. \ds\f{dc_{0l}}{dl}\right|_{l=l_0}\\
=&-\left[\mathcal S_3\left(c_{0}\right)+c_{0}\mathcal S_3'(c_{0})\right]\left. \ds\f{dc_{0l}}{dl}\right|_{l=l_0},\\
\end{split}
\end{equation*}
where we have used $c_{0l_0}=c_0$ in the last step. Noting that $\mathcal S_3(c_0)=0$, we see from \eqref{cc0l} and \eqref{mons3} that $\mathcal S'(l_0)>0$.
Note from Theorem \ref{thglobal} that $\mathcal L=[ \tilde l+\epsilon,1/\epsilon]$ with $0<\epsilon\ll1$. Then, for sufficiently small $\epsilon$, $l_0\in {int}( \mathcal L)=(\tilde l+\epsilon,1/\epsilon)$ and (ii) holds. This completes the proof.
\end{proof}

By Lemma \ref{suppst}, we can solve \eqref{H} for $\la=0$.
\begin{lemma}\label{l5}
Suppose  that $\la=0$, and let $\mathcal T (\alpha)$ be defined in Lemma \ref{suppst}. Then the following statements hold:
\begin{enumerate}
\item [$\rm{(i)}$] If $\mathcal T (\alpha)<0$, then \eqref{H} has no solution;
\item [$\rm{(ii)}$] If $\mathcal T (\alpha)>0$, then \eqref{H}
 has  a unique solution $$(\delta,s_1,s_2,\nu, w, z,l)=(\delta_0,s_{10},s_{20},\nu_0, w_0, z_0,l_0),$$ where
 $l_0$ is obtained in Lemma \ref{suppst},
\begin{equation*}
\begin{split}
&s_{10}=0,\;w_0= 0,\; z_0= 0,\;\nu_0=\sqrt{ -\frac{  \int_\Omega e^{\alpha m(x)} M_{2}^{(0,l_0)} dx   \int_\Omega M_{3}^{(0,l_0)} dx    }{ | \Omega |  \int_\Omega e^{\alpha m(x)} dx }   },\\
&\delta_0=   \sqrt{ \ds \frac{1 }{1+\left(\f{\int_\Omega M_{3}^{(0,l_0)} dx }{\nu_0|\Omega|} \right)^2}},\;\;s_{20}=-\ds\f{\delta_0}{\nu_0|\Omega|}\ds \int_\Omega  M_{3}^{(0,l_0)} dx,
\end{split}
\end{equation*}
and $M_{i}^{(0,l)}$ is defined in \eqref{Mi0} for $i=1,2,3,4$.
\end{enumerate}
\end{lemma}
\begin{proof}
It follows from \eqref{Mi0} that
\begin{equation}\label{signs}
\ds\int_\Omega e^{\alpha m(x)}M_{2}^{(0,l_0)}dx<0,\;\;\ds\int_\Omega M_{3}^{(0,l_0)} dx>0,
\end{equation}
which implies that  $\nu_0$ is well defined.

Substituting $\la =0$ into $ h_{2}=0$ and $h_{4}=0$, respectively, we have $ w= w_0= 0$ and $ z= z_0= 0$. Note from the second equation of
\eqref{solcq2} that
\begin{equation}\label{M4}
\int_\Omega  M_{4}^{(0,l)} dx=0 \;\;\text{for any}\;\;l\in\mathcal L.
\end{equation}
Then
plugging $ w= z=0$ and $\lambda=0$ into $ h_i=0$ for $i=1,3,5$, respectively, we have
\begin{equation}\label{c1}
\left( {\begin{array}{*{20}{c}}
\int_\Omega e^{\alpha m(x)} M_{1}^{(0,l)} dx-{\rm i}\nu\int_\Omega  e^{\alpha m(x)}dx &     \int_\Omega e^{\alpha m(x)} M_{2}^{(0,l)} dx\\
\int_\Omega M_3^{(0,l)} dx &  -{\rm i}\nu | \Omega|
\end{array}} \right)\left( {\begin{array}{*{20}{c}}
 \delta\\
s_{1} + {\rm i} s_{2}
\end{array}} \right)= \left( {\begin{array}{*{20}{c}}
0\\
0
\end{array}} \right),
\end{equation}
and
\begin{equation}\label{c1si}
\delta^2 +s_1^2+s_2^2=1.
\end{equation}
Therefore, \eqref{H} has a solution if and only if \eqref{c1}-\eqref{c1si} is solvable for some value of $(\delta,s_1,s_2,\nu,l)$ with
$\delta,\nu\ge0$, $s_1,s_2\in \mathbb R$ and $l\in \mathcal L$.
It follows from \eqref{signs} that \eqref{c1}-\eqref{c1si} is solvable (or \eqref{H} is solvable) if and only if
\begin{equation}\label{Hocond}
\mathcal S(l)=\int_\Omega e^{\alpha m(x)} M_{1}^{(0,l)} dx=0
\end{equation}
is solvable for some $l\in\mathcal L$.
From Lemma \ref{suppst}, we see that if $\mathcal T (\alpha)<0$, then \eqref{Hocond} has no solution in $\mathcal L$; and if $\mathcal T (\alpha)>0$, then \eqref{Hocond} has a unique solution $l_0$ in $\mathcal L$ with $\mathcal S(l_0)=0$.

Substituting $l=l_0$ into \eqref{c1}, we compute that
\begin{equation}\label{nus12}
\begin{split}
 s_1=s_{10}=0,\; \nu=\nu_0,\;s_2=-\ds\f{\delta}{\nu_0|\Omega|}\ds \int_\Omega  M_{3}^{(0,l_0)} dx.\\
\end{split}
\end{equation}
Then, plugging \eqref{nus12} into \eqref{c1si}, we get
$\delta=\delta_0$ and $s_2=s_{20}$.
This completes the proof.
\end{proof}

Now, we solve \eqref{H} for $\la>0$.
\begin{theorem}\label{the1}
Suppose that $\mathcal T (\alpha)>0$, where $\mathcal T (\alpha)$ is defined in Lemma \ref{suppst}.
Then there exists $\tilde\la\in(0,\delta _\epsilon)$, where $\delta _\epsilon$ is obtained in Theorem \ref{thglobal}, and a continuously differentiable mapping
$$\la\mapsto(\delta_{\la},s_{1\la},s_{2\la}, \nu_{\la},w_{\la}, z_\la,l_\la):
[0,\tilde\la]\to\mathbb R^4\times \left(\left(X_{1}\right)_{\mathbb C}\right)^2\times  \mathcal L$$ such that
\eqref{H}
 has a unique solution  $(\delta_{\la},s_{1\la},s_{2\la}, \nu_{\la},w_{\la}, z_\la,l_\la)$ for $\la\in[0,\tilde\la]$, and
\begin{equation*}
(\delta_{\la},s_{1\la},s_{2\la}, \nu_{\la},w_{\la}, z_\la,l_\la)=(\delta_{0},s_{10},s_{20}, \nu_0,w_0, z_0,l_0)
\end{equation*}
 for $\la=0$,
where $(\delta_{0},s_{10},s_{20}, \nu_0,w_0, z_0,l_0)$ is defined in Lemma \ref{l5}.
\end{theorem}

\begin{proof}
We first show the existence.
It follows from Lemma \ref{l5} that
$ H\left(K_0\right)= 0$,
where $K_0=(\delta_{0},s_{10},s_{20}, \nu_0,w_0, z_0,l_0,0)$.
Note from \eqref{M4} that
\begin{equation}\label{suppM}
\ds\int_\Omega M_4^{(0,l)}dx=0,\;\; \ds\f{d}{dl}\left(\ds\int_\Omega M_4^{(0,l)}dx\right)=0\;\;\text{for all}\;\; l\in\mathcal L.
\end{equation}
Then
the Fr\'echet derivative of $H (\delta, s_1,s_2, \nu,w, z,l,\la)$ with respect to $(\delta,s_1,s_2, \nu,w, z, l)$ at $K_0$  is as follows:
\begin{equation*}
K( \hat\delta,\hat s_1,\hat s_2, \hat \nu,\hat w,\hat z,\hat l):=(k_1,k_2,k_3,k_4,k_5)^T,
\end{equation*}
where $(\hat\delta,\hat s_1,\hat s_2,\hat \nu,\hat w,\hat z,\hat l)\in \mathbb R^4\times \left(\left(X_{1}\right)_{\mathbb C}\right)^2\times \mathcal L$, and
\begin{equation*}
\begin{split}
k_1=&\ds \int_\Omega  { e^{\alpha m(x)} \left[ \left(M_{1}^{(0,l_0)}-{\rm i} \nu_0\right)(\hat \delta + \hat w )+ M_{2}^{(0,l_0)} (\hat s_1 + {\rm i}\hat s_2+\hat z)  \right]} dx- {\rm i}\hat \nu \delta_0 \int_\Omega e^{\alpha m(x)} dx\\
&+\hat l \left . \f{d}{dl}\left[ \int_\Omega e^{\alpha m(x)} \left( \delta_0  M_{1}^{(0,l)} +{\rm i}s_{20}  M_{2}^{(0,l)}  \right) dx \right]\right|_{l=l_0},\\
k_{2}=&\ds L \hat w,\\
k_3=& \int_\Omega  {  \left[ M_{3}^{(0,l_0)}(\hat \delta + \hat w )+ M_{4}^{(0,l_0)} \hat z \right]} dx+ \hat\nu s_{20} | \Omega | -{\rm i} \nu_0 (\hat s_1 + {\rm i}\hat s_2)| \Omega | \\
&+\hat l \delta_0  \left .  \f{d}{dl} \left(\int_\Omega M_{3}^{(0,l)} dx  \right)\right|_{l=l_0},\\
k_{4}=&\ds \theta\Delta\hat z ,\\
k_5=& 2 \delta_0| \Omega |\hat \delta+2s_{20}| \Omega |\hat s_2,\\
\end{split}
\end{equation*}
where we have used \eqref{suppM} to obtain $k_3$.
If $ K( \hat \delta,\hat s_1,\hat s_2,\hat \nu,\hat w,\hat z,\hat l)= 0$, then $ \hat w= 0$ and $\hat z= 0$. Substituting $\hat w=\hat z= 0$ into $k_1=0$ and $k_3=0$, respectively, separating the real and imaginary parts, and noting that $k_5=0$, we get
\begin{equation*}
(D_{ij})(\hat \delta, \hat s_1,\hat s_2,\hat \nu,\hat l)^T= (0,0,0,0,0)^T.
\end{equation*}
Here $D_{ij}=0$ except
\begin{equation*}
\begin{split}
&D_{12}= \int_\Omega e^{\alpha m(x)} M_{2}^{(0,l_0)} dx,\;\;D_{15}=  \delta_0  \mathcal S'(l_0),\;\;D_{21}=- \nu_0 \int_\Omega e^{\alpha m(x)} dx, \;\;\;D_{34}=s_{20}| \Omega |,\\
&D_{23}= \int_\Omega e^{\alpha m(x)} M_{2}^{(0,l_0)} dx,\;\;D_{24}=- \delta_0 \int_\Omega e^{\alpha m(x)} dx,\\
&D_{25}= s_{20}  \f{d}{dl} \left .  \left(\int_\Omega  e^{\alpha m(x)} M_{2}^{(0,l)}  dx\right)\right|_{l=l_0},\;\;D_{31}=  \int_\Omega  M_{3}^{(0,l_0)} dx,\;\;D_{33}= \nu_0| \Omega |,\\
&D_{35}=\delta_0 \f{d}{dl} \left .\left(  \int_\Omega M_{3}^{(0,l)} dx \right) \right|_{l=l_0},\;\;D_{42}= -\nu_0 | \Omega |,\;\;D_{51}=2 \delta_0| \Omega |,\;\;D_{53}=2s_{20}| \Omega |,
\end{split}
\end{equation*}
where $\mathcal S(l)$ is defined in Lemma \ref{suppst}. 
A direct computation implies that
\begin{equation}\label{determind}
\begin{split}
\left|\left(D_{ij}\right)\right|=&2\nu_0 | \Omega |^2\delta_0\mathcal S'(l_0) \left|\begin{array}{ccc}
- \nu_0 \int_\Omega e^{\alpha m(x)} dx& \int_\Omega e^{\alpha m(x)} M_{2}^{(0,l_0)} dx&- \delta_0 \int_\Omega e^{\alpha m(x)} dx \\
 \int_\Omega  M_{3}^{(0,l_0)} dx&\nu_0| \Omega |&s_{20}| \Omega |\\
\delta_0&s_{20}&0
\end{array}\right|\\
=&2\nu_0   \delta_0  | \Omega |^2 \mathcal S'(l_0) \bigg (  \delta_0  s_{20} |\Omega| \int_\Omega e^{\alpha m(x)} M_{2}^{(0,l_0)} dx+\delta_0^2  \nu_0  |\Omega| \int_\Omega  e^{\alpha m(x)} dx \\
&+\nu_0 s_{20}^2  |\Omega|  \int_\Omega  e^{\alpha m(x)} dx -\delta_0  s_{20}  \int_\Omega  e^{\alpha m(x)} dx \int_\Omega  M_{3}^{(0,l_0)} dx \bigg).
\end{split}
\end{equation}
Since $(\delta_0,s_{10},s_{20},\nu_0)$ satisfies \eqref{c1}, we have
\begin{equation}\label{snu}
s_{20}  \int_\Omega e^{\alpha m(x)} M_{2}^{(0,l_0)} dx = \nu_0 \delta_0 \int_\Omega  e^{\alpha m(x)} dx,\;\;\delta_0  \int_\Omega M_{3}^{(0,l_0)} dx =-\nu_0 s_{20} |\Omega|.
\end{equation}
Plugging \eqref{snu} into \eqref{determind}, we have
\begin{equation*}
|(D_{ij})|=4\nu_0^2   \delta_0  | \Omega |^3\mathcal S'(l_0)\left(  \delta_0^2+s_{20}^2 \right) \int_\Omega  e^{\alpha m(x)} dx>0,
\end{equation*}
where we have used Lemma \ref{suppst} (ii) in the last step.
Therefore, $\hat \delta=0,\hat  s_1=0,\hat  s_2=0,\hat \nu=0$ and $\hat l=0$, which implies that $ K$ is injective and thus bijective.
From the implicit function theorem, we see that there exists $\tilde\la\in(0,\delta _\epsilon)$
and a continuously differentiable mapping
$$\la\mapsto(\delta_{\la},s_{1\la},s_{2\la},\nu_{\la}, w_{\la}, z_\la,l_\la):
[0,\tilde\la]\to\mathbb R^4\times \left(\left(X_{1}\right)_{\mathbb C}\right)^2\times \mathcal L$$
such that $ H\left(\delta_{\la},s_{1\la},s_{2\la}, \nu_\la,w_\la, z_\la,l_\la,\la\right)= 0$, and  for $\la=0$,
\begin{equation*}
\left(\delta_{\la},s_{1\la},s_{2\la}, \nu_{\la},w_{\la}, z_\la,l_\la \right)=(\delta_0,s_{10},s_{20},\nu_0, w_0, z_0,l_0).
\end{equation*}

Now, we show the uniqueness. From the implicit function theorem, we only need to verify that if $\left(\delta^{\la},s_{1}^{\la},s_{2}^{\la}, \nu^{\la},w^{\la},z^{\la},l^{\la}\right)$ is a solution of \eqref{H}  for $\la\in(0,\tilde\la]$,
then
\begin{equation}\label{limv}
\lim_{\la\to0}\left(\delta^{\la},s_{1}^{\la},s_{2}^{\la},\nu^{\la}, w^{\la},z^{\la},l^{\la}\right)= (\delta_0,s_{10},s_{20}, \nu_0,w_0, z_0,l_0) \;\;\text{in}\;\;\mathbb R^4\times \left((X_1)_{\mathbb C}\right)^2\times \mathbb R.
\end{equation}
Noticing that $h_5\left(\delta^{\la},s_{1}^{\la},s_{2}^{\la},\nu^{\la}, w^{\la},z^{\la},l^{\la}\right)=0$, we see that $|\delta^{\la}|$, $|s_{1}^{\la}|$, $|s_{2}^{\la}|$, $\| w^{\la}\|_2$ and $\| z^{\la}\|_2$ are bounded for $\la\in[0,\tilde\la]$. Since $l^\la\in\mathcal L$, we obtain that $l^{\la}$ is bounded for $\la\in[0,\tilde\la]$. Moreover,
$|\nu^\la|$ is also bounded for $\la\in[0,\tilde\la]$ from Lemma \ref{l3}.
Let
\begin{equation*}
\begin{split}
 \mathcal{Q}_1(\la)=&\la e^{\alpha m(x)}  \left[ M_{1}^{(\la,l^{\la})} (\delta^{\la}+w^{\la}) + M_{2}^{(\la,l^{\la})} (s_1^{\la} + {\rm i} s_2^{\la}+z^{\la})  \right]\\
 &-{\rm i}\la \nu^{\la} e^{\alpha m(x)}  (\delta^{\la}+w^{\la})-\frac{\la }{ | \Omega |} h_1\left(\delta^{\la},s_{1}^{\la},s_{2}^{\la},\nu^{\la}, w^{\la},z^{\la},l^{\la}\right),\\
 \mathcal{Q}_2(\la)=&\la \left[ M_{3}^{(\la,l^{\la})} (\delta^{\la}+w^{\la}) + M_{4}^{(\la,l^{\la})} (s_1^{\la} + {\rm i} s_2^{\la}+z^{\la}) \right]\\
 &-{\rm i}\la \nu^{\la}  (s_1^{\la} + {\rm i} s_2^{\la}+z^{\la})-\frac{\la }{ | \Omega |} h_3\left(\delta^{\la},s_{1}^{\la},s_{2}^{\la},\nu^{\la}, w^{\la},z^{\la},l^{\la}\right).
\end{split}
\end{equation*}
Noting that $\left(\delta^{\la},s_{1}^{\la},s_{2}^{\la}, \nu^{\la},w^{\la},z^{\la},l^{\la}\right)$ is bounded in $\mathbb R^4\times \left((Y_1)_{\mathbb C}\right)^2\times \mathbb R$,
 we see from \eqref{bondM} that $\lim_{\la\to0} \mathcal{Q}_1(\la)=0$ and $\lim_{\la\to0} \mathcal{Q}_2(\la)=0$ in $(Y_1)_{\mathbb C}$.
Since
\begin{equation*}
w^{\la}=-L^{-1}\left[ \mathcal{Q}_1(\la)\right],\;\;z^{\la}=-\Delta^{-1}\left[ \mathcal{Q}_2(\la)\right],
\end{equation*}
 where $L^{-1}$ and $\Delta^{-1}$ are bounded operators from $\left(Y_{1}\right)_{\mathbb C}$ to $\left(X_{1}\right)_{\mathbb C}$, we get
\begin{equation*}
\lim_{\la\to0}   w^{\la}= w_0=0,\;\;\lim_{\la\to0}   z^{\la}= z_0=0\;\;\text{in}\;\;\left(X_1\right)_{\mathbb C}.
\end{equation*}
Since $\left(\delta^{\la},s_{1}^{\la},s_{2}^{\la}, \nu^{\la},l^{\la}\right)$ is bounded in $\mathbb R^5$ for $\la\in(0,\tilde\la]$, we see that, up to a subsequence,
\begin{equation*}
\lim_{\la\to0}   \delta^{\la}= \delta^*, \;\;\lim_{\la\to0}   s^{\la}_1= s_1^*,\;\;\lim_{\la\to0}   s^{\la}_2= s_2^*,\; \;\lim_{\la\to0}  \nu^{\la}= \nu^*,\;\;\lim_{\la\to0}  l^{\la}= l^*.
\end{equation*}
Taking $\la\to0$ on both sides of $$H\left(\delta^{\la},s_{1}^{\la},s_{2}^{\la},\nu^{\la}, w^{\la},z^{\la},l^{\la}\right)=0,$$ we see that $H\left(\delta^{*},s_{1}^{*},s_{2}^{*},\nu^{*}, 0,0,l^{*}\right)=0$.
This combined with Lemma \ref{l5} implies that $\delta^*=\delta_0$,  $s_1^*=s_{10}$, $s_2^*=s_{20}$, $\nu^*=\nu_0$ and $l^*=l_0$. This completes the proof.
\end{proof}

Then from Lemma \ref{lema32} and Theorem \ref{the1}, we have the following result.
\begin{theorem}\label{the1*}
Let $\mathcal L$ be defined in Theorem \ref{thglobal}. Suppose that $\mathcal T (\alpha)>0$ and $\la\in(0,\tilde\la]$, where $0<\tilde\la\ll 1$, and $\mathcal T (\alpha)$ is defined in Lemma \ref{suppst}. Then $(\varphi,\psi,\sigma,l)$ solves
\begin{equation*}
\begin{cases}
\left(A_l(\la)- {\rm i}\sigma I\right)( \varphi, \psi)^T= 0,\\
\sigma\geq0,\;l\in\mathcal L,\;(\varphi, \psi)^T\ne (0,0)^T\in X^2_{\mathbb C},\\
\end{cases}
\end{equation*}
 if and only if
$$ \sigma=\la\nu_\la,\;\varphi=\kappa \varphi_\la =\kappa(\delta_\la+ w_\la),\;\psi=\kappa\psi_\la=\kappa[ (s_{1\la}+{\rm i}s_{2\la})+ z_\la],\;l=l_\la,$$
where $A_l(\la)$ is defined in \eqref{Al}, $I$ is the identity operator, $\kappa\in\mathbb C$ is a nonzero constant, and $(\delta_\la,s_{1\la},s_{2\la}, \nu_{\la},w_{\la}, z_\la, l_\la)$ is obtained in Theorem \ref{the1}.
\end{theorem}

To show that ${\rm i}\la\nu_\la$ is a simple eigenvalue of $A_{l_\la}(\la)$, we need to consider the following operator
\begin{equation}\label{AlT}
A^H(\la) : = \left( {\begin{array}{cc}
 L&0\\
0&\theta \Delta
\end{array}} \right) + \la\left( {\begin{array}{cc}
e^{\alpha m(x)}M_1^{(\la,l_\la)}+{\rm i}\nu_\la e^{\alpha m(x)}&M_3^{(\la,l_\la)}\\
e^{\alpha m(x)}M_2^{(\la,l_\la)}& M_4^{(\la,l_\la)}+{\rm i}\nu_\la
\end{array}} \right).
\end{equation}
Let
\begin{equation}\label{E}
\mathcal{I} : = \left( {\begin{array}{cc}
 e^{\alpha m(x)}&0\\
0&1
\end{array}} \right).
\end{equation}
Then $A^H(\la)$ is the adjoint operator of $\mathcal{I}\left( A_{l_\la}(\la)- {\rm i}\la \nu_\la I \right) $. That is, for any $(\phi_1,\phi_2)^T\in X_{\mathbb C}$ and $(\psi_1,\psi_2)^T\in X_{\mathbb C}$, we have
\begin{equation}\label{adjoint}
\left\langle A^H(\la)\left(\phi_1,\phi_2\right)^T,\left(\psi_1,\psi_2\right)^T\right\rangle=\left\langle \left(\phi_1,\phi_2\right)^T,\mathcal{I}\left( A_{l_\la}(\la)- {\rm i}\la \nu_\la I \right)\left(\psi_1,\psi_2\right)^T\right\rangle.
\end{equation}

\begin{lemma}\label{ad}
Let $A^H(\la)$ be defined in \eqref{AlT}. Suppose that $\mathcal T (\alpha)>0$ and $\la\in(0,\tilde\la]$, where $0<\tilde\la\ll 1$ and $\mathcal T (\alpha)$ is defined in Lemma \ref{suppst}. Then
$$\mathcal N \left[ A^H(\la) \right] ={\rm span}  [( \tilde\varphi_\la,\tilde \psi_\la)^T ],$$
and, ignoring a scalar factor, $(\tilde\varphi_\la,\tilde \psi_\la)$ can be represented as
\begin{equation}\label{trans-vps}
\begin{cases}
\tilde\varphi_\la =\tilde \delta_\la+\tilde w_\la, \;\;\text{where}\; \tilde \delta\geq0\;\;\text{and}\;\;\tilde w_\la\in \left(X_{1}\right)_{\mathbb C},\\
\tilde \psi_\la  =(\tilde s_{1\la}+{\rm i}\tilde s_{ 2\la}) + \tilde z_\la,\;\;\text{where}\;\; \tilde s_{1\la},\tilde s_{2\la}\in\mathbb R\;\;\text{and}\;\;\tilde z_\la \in \left(X_{1}\right)_{\mathbb C},\\
\| \tilde\varphi_\la\|_2^{2}+\|\tilde \psi_\la \|_2^{2}=| \Omega |.
\end{cases}
\end{equation}
Moreover,
$
\lim_{\la\to0}\left(\tilde \delta_\la,\tilde s_{1\la},\tilde s_{2\la},\tilde w_{\la},\tilde z_\la \right)=(\tilde \delta_0,\tilde s_{10},\tilde s_{20},\tilde w_0,\tilde z_0 )$ in $\mathbb R^3\times \left((X_1)_{\mathbb C}\right)^2$,
where $\tilde w_0= 0,\;\tilde z_0= 0,$
\begin{equation*}
\tilde \delta_0=\sqrt{\frac{1}{1+\left(\frac{   \nu_0   \int_\Omega e^{\alpha m(x)} dx }{ \int_\Omega  M_3^{(0,l_0)} dx}\right)^2 }},
\;\;\tilde s_{10}=0,
\;\;\tilde s_{20}=-\frac{ \tilde \delta_{0}  \nu_0   \int_\Omega e^{\alpha m(x)} dx }{ \int_\Omega  M_3^{(0,l_0)} dx},
\end{equation*}
and $\nu_0,l_0$ are defined in Lemma \ref{l5}.
\end{lemma}

\begin{proof}
It follows from Theorem \ref{the1*} that $0$ is an eigenvalue of $ \mathcal{I}\left( A_{l_\la}(\la)- {\rm i}\la \nu_\la I \right)$ and $\mathcal N \left[\mathcal{I}\left( A_{l_\la}(\la)- {\rm i}\la \nu_\la I \right) \right]$ is one-dimensional
for $\la\in(0,\tilde\la]$, where $\mathcal I$ is defined in \eqref{E}. Then
$0$ is also an eigenvalue of $ A^H(\la)$ and $\mathcal N [ A^H(\la)]$ is also one-dimensional. Then  $\mathcal N [ A^H(\la)]={\rm span}  [(\tilde\varphi_\la,\tilde \psi_\la)^T]$, and $( \tilde\varphi_\la,\tilde \psi_\la)$ can be represented as \eqref{trans-vps}.
By the similar arguments as in the proof of Theorem \ref{the1}, we see that
\begin{equation*}
\lim_{\la\to0}\tilde w_\la=0,\;\;\lim_{\la\to0}\tilde z_\la=0\;\;\text{in}\;\;\left(X_1\right)_{\mathbb C},
\end{equation*}
and up to a subsequence,
\begin{equation*}
\lim_{\la\to0}   \tilde \delta_{\la}=\delta_0^*,\;\;\lim_{\la\to0}   \tilde s_{1\la}= s_{10}^*,\;\;\lim_{\la\to0}   \tilde s_{2\la}=s_{20}^*,
 \end{equation*}
where $\left( \delta_0^{*}, s_{10}^*, s_{20}^*\right)$ satisfies $\left(\delta_0^*\right)^2+\left( s_{ 10}^*\right)^2+\left( s_{ 20}^*\right)^2=1$, and
\begin{equation*}
\left( {\begin{array}{*{20}{c}}
 \ds {\rm i}\nu_0  \int_\Omega e^{\alpha m(x)} dx&    \ds \int_\Omega  M_3^{(0,l_0)} dx\\
\ds \int_\Omega e^{\alpha m(x)} M_2^{(0,l_0)} dx &  {\rm i}\nu_0 | \Omega|
\end{array}} \right)\left( {\begin{array}{*{20}{c}}
\delta_{0}^*\\
 s_{10}^* + {\rm i}  s_{20}^*
\end{array}} \right)= \left( {\begin{array}{*{20}{c}}
0\\
0
\end{array}} \right).
\end{equation*}
A direct computation implies that
$\left( \delta_0^{*}, s_{10}^*, s_{20}^*\right)=\left(\tilde \delta_0, \tilde s_{10}, \tilde s_{20}\right)$,
and consequently, $\lim_{\la\to0}\left(\tilde \delta_\la,\tilde s_{1\la},\tilde s_{2\la}\right)=\left(\tilde \delta_0, \tilde s_{10}, \tilde s_{20}\right)$.
This completes the proof.
\end{proof}

Then, by virtue of Lemma \ref{ad}, we show that ${\rm i} \la \nu_\la$ is simple.

\begin{theorem}\label{sim}
Suppose that $\mathcal T (\alpha)>0$ and $\la\in(0,\tilde\la]$, where $0<\tilde\la\ll 1$ and $\mathcal T (\alpha)$ is defined in Lemma \ref{suppst}. Then ${\rm i} \la \nu_\la$ is a simple eigenvalue of $ A_{l_\la}(\la) $, where $ A_l(\la) $ is defined in \eqref{Al}.
\end{theorem}

\begin{proof}
It follows from Theorem \ref{the1*} that $\mathcal N [ A_{l_\la}(\la) -  {\rm i}\la \nu_\la I] ={\rm span}  [( \varphi_\la,  \psi_\la)^T]$, where $ \varphi_\la$ and $\psi_\la$ are defined in Theorem \ref{the1}. Then we show that
\begin{equation*}
\mathcal N [ A_{l_\la}(\la) -  {\rm i}\la \nu_\la I]^2 =\mathcal N [ A_{l_\la}(\la) -  {\rm i}\la \nu_\la I].
\end{equation*}
Letting $( \Psi_1,  \Psi_2)^T \in \mathcal N [ A_{l_\la}(\la) -  {\rm i}\la \nu_\la I]^2$, we have
\begin{equation*}
 [ A_{l_\la}(\la) -  {\rm i}\la \nu_\la I] ( \Psi_1,  \Psi_2)^T  \in \mathcal N [ A_{l_\la}(\la) -  {\rm i}\la \nu_\la I]={\rm span}  [( \varphi_\la,  \psi_\la)^T],
\end{equation*}
and consequently, there exists a constant $s\in\mathbb C$ such that
\begin{equation*}
 [ A_{l_\la}(\la) -  {\rm i}\la \nu_\la I] ( \Psi_1,  \Psi_2)^T  =s ( \varphi_\la,  \psi_\la)^T.
\end{equation*}
Then
\begin{equation}\label{Psi2}
\mathcal I[ A_{l_\la}(\la) -  {\rm i}\la \nu_\la I] ( \Psi_1,  \Psi_2)^T  =s\mathcal I( \varphi_\la,  \psi_\la)^T,
\end{equation}
where $\mathcal I$ is defined in \eqref{E}. Note that $A^H(\la)$ is the adjoint operator of  $\mathcal I[ A_{l_\la}(\la) -  {\rm i}\la \nu_\la I]$, and it follows from Lemma \ref{ad} that
$$\mathcal N \left[ A^H(\la) \right] ={\rm span}  [( \tilde\varphi_\la,\tilde \psi_\la)^T ].$$
Then, by \eqref{adjoint} and \eqref{Psi2}, we have
\begin{equation*}
\begin{split}
0&=\left\langle  A^H(\la)(\tilde\varphi_\la, \tilde \psi_\la)^T ,( \Psi_1,  \Psi_2)^T  \right\rangle=\left\langle (\tilde\varphi_\la,\tilde \psi_\la)^T, \mathcal{I}\left( A_{l_\la}(\la)- {\rm i}\la \nu_\la I \right) ( \Psi_1,  \Psi_2)^T    \right\rangle \\
&=s \mathcal W(\la),
\end{split}
\end{equation*}
where
\begin{equation}\label{Mla}
\mathcal W(\la):=\left\langle (\tilde\varphi_\la,\tilde \psi_\la)^T, \mathcal I( \varphi_\la,  \psi_\la)^T\right\rangle =\int_\Omega \left( e^{\alpha m(x)}\overline{\tilde\varphi_\la } \varphi_\la +  \overline{ \tilde\psi_\la } \psi_\la \right) dx.
\end{equation}
It follows from Lemmas \ref{l5} and \ref{ad} and Theorem \ref{the1} that
\begin{equation}\label{ptilde}
\lim_{\la\to 0}\mathcal W(\la)=2\tilde \delta_0 \delta_0 \int_\Omega e^{\alpha m(x)} dx>0,
\end{equation}
which implies that $s=0$ for $\la\in(0,\tilde\la]$ with $0<\tilde\la\ll1$. Therefore,
\begin{equation*}
\mathcal N [ A_{l_\la}(\la) -  {\rm i}\la \nu_\la I]^2 =\mathcal N [ A_{l_\la}(\la) -  {\rm i}\la \nu_\la I].
\end{equation*}
This completes the proof.
\end{proof}

Note that ${\rm i}\la\nu_\la$ is simple if exists. Then, by using the implicit function
theorem, we see that there exists a neighborhood $P_\la\times V_\la \times O_\la$ of $ ( \varphi_\la, \psi_\la, {\rm i}\nu_\la,l_\la)$ ($P_\la$, $V_\la$ and $O_\la$ are neighborhoods of $( \varphi_\la, \psi_\la)$, ${\rm i}\nu_\la$ and $l_\la$, respectively), and a continuously differentiable function  $ ( \varphi(l), \psi(l), \mu(l)):O_\la\to P_\la\times V_\la$ such that $\mu(l_\la)={\rm i}\nu_\la,\;( \varphi(l_\la), \psi(l_\la))=( \varphi_\la, \psi_\la)$, where $\nu_\la$, $\varphi_\la$ and $\psi_\la$ are defined in Theorem \ref{the1}. Moreover, for each $l \in O_\la$, the only eigenvalue of $A_{l}(\la)$ in $V_\la$ is $\mu(l)$, and
\begin{equation}\label{tra}
\left(A_{l}(\la)- \mu(l) I\right)( \varphi(l), \psi(l))^T= 0.
\end{equation}
Then, we show that the following transversality condition holds.

\begin{theorem}\label{the2}
Let $l_\la$ be obtained in Theorem \ref{the1}. Then
\begin{equation*}
 \left.\frac{ d \mathcal Re [\mu(l)]}{ d l}\right|_{l=l_\la} > 0.
\end{equation*}
\end{theorem}

\begin{proof}
Multiplying both sides of \eqref{tra} by $\mathcal I$ to the left, and differentiating the result with respect to $l$ at $l=l_\la$, we have
\begin{equation}\label{tra1}
\begin{split}
\left.\frac{ d \mu}{ d l}\right|_{l=l_\la}  \mathcal{I} ( \varphi_\la, \psi_\la)^T= & \mathcal{I} \left(A_{l_\la}(\la)- {\rm i}\nu_\la I\right)\left.\left(\frac{d \varphi }{ d l}, \frac{d \psi }{ d l}  \right)^T\right|_{l=l_\la}\\
&+\mathcal I\left.\left( {\begin{array}{cc}
\ds\frac{dM^{(\la,l)}_1}{d l}&\ds\frac{dM^{(\la,l)}_2}{d l}\\
\ds\frac{dM^{(\la,l)}_3}{d l}&\ds\frac{dM^{(\la,l)}_4}{d l}
\end{array}} \right)\right|_{l=l_\la} (\varphi_\la, \psi_\la)^T,
\end{split}
\end{equation}
where $M^{(\la,l)}_i\;(i=1,\ldots,4)$ and $\mathcal I$ are defined in \eqref{Mi} and \eqref{E}, respectively.
Note from \eqref{adjoint} and Lemma \ref{ad} that
\begin{equation*}
\begin{split}
&\left\langle(\tilde\varphi_\la, \tilde \psi_\la)^T , \mathcal{I} \left(A_{l_\la}(\la)- {\rm i}\nu_\la I\right)\left.\left(\frac{d \varphi }{ d l}, \frac{d \psi }{ d l}  \right)^T\right|_{l=l_\la}\right\rangle\\
&=\left\langle A^H(\la)(\tilde\varphi_\la, \tilde \psi_\la)^T, \left.\left(\frac{d \varphi }{ d l}, \frac{d \psi }{ d l}  \right)^T\right|_{l=l_\la} \right\rangle=0,
\end{split}
\end{equation*}
where $\tilde\varphi_\la$ and $\tilde \psi_\la$ are defined in Lemma \ref{ad}.
Then, multiplying both sides of \eqref{tra1} by $(\tilde\varphi_\la, \tilde \psi_\la )$ to the left, and integrating the result over $\Omega$, we
have
\begin{equation}\label{ptilde2}
\begin{split}
\mathcal W(\la)\left.\frac{ d \mu}{ d l}\right|_{l=l_\la}
=&\la\int_\Omega e^{\alpha m(x)} \overline{\tilde\varphi_\la}\varphi_\la \left .  \frac{d M_{1}^{(\la,l)} }{d l} \right|_{l=l_\la} dx+ \la \int_\Omega \overline {\tilde \psi_\la}\varphi_\la\left .  \frac{d M_{3}^{(\la,l)} }{d l} \right|_{l=l_\la} dx\\
& +\la\int_\Omega e^{\alpha m(x)}  \overline{\tilde \varphi_\la}\psi_\la \left .  \frac{d M_{2}^{(\la,l)} }{d l} \right|_{l=l_\la} dx+\la \int_\Omega \overline{\tilde\psi_\la}\psi_\la \left .  \frac{d M_{4}^{(\la,l)} }{d l} \right|_{l=l_\la} dx,
\end{split}
\end{equation}
where $\mathcal W(\la)$ is defined in \eqref{Mla}.
From Theorem \ref{the1} and Lemma \ref{ad}, we see  that
\begin{equation}\label{psi00}
\lim_{\la\to0}\left(l_\la,\varphi_\la,\psi_\la,\tilde \varphi_\la,\tilde\psi_\la \right)= \left(l_0,\delta_0,{\rm i}s_{20},\tilde \delta_0,{\rm i}\tilde s_{20}\right) \;\;\text{in}\;\; \mathbb{R}\times X_{\mathbb C}^4.
\end{equation}
It follows from Theorem \ref{thglobal} that
$( u^{(\la,l)}, v^{(\la,l)})$ is continuously differentiable. This, combined with the embedding theorems and Eq.  \eqref{psi00}, implies that
\begin{equation}\label{limmulti}
\begin{split}
&\lim_{\la\to0}\int_\Omega e^{\alpha m(x)} \overline{\tilde\varphi_\la}\varphi_\la \left .  \frac{d M_{1}^{(\la,l)} }{d l} \right|_{l=l_\la} dx=\delta_0\tilde \delta_0\left. \ds\f{d}{dl}\left(\int_\Omega e^{\alpha m(x)} M_{1}^{(0,l)}   dx\right)\right|_{l=l_0},\\
&\lim_{\la\to0} \int_\Omega e^{\alpha m(x)}  \overline{\tilde \varphi_\la}\psi_\la \left .  \frac{d M_{2}^{(\la,l)} }{d l} \right|_{l=l_\la} dx={\rm i}
\tilde\delta_0s_{20}\left. \ds\f{d}{dl}\left(\int_\Omega e^{\alpha m(x)} M_{2}^{(0,l)}   dx\right)\right|_{l=l_0},\\
& \lim_{\la\to0}\int_\Omega \overline {\tilde \psi_\la}\varphi_\la\left .  \frac{d M_{3}^{(\la,l)} }{d l} \right|_{l=l_\la} dx=-{\rm i}\delta\tilde s_{20}\left. \ds\f{d}{dl}\left(\int_\Omega  M_{3}^{(0,l)}   dx\right)\right|_{l=l_0},\\
&\lim_{\la\to0}\int_\Omega \overline{\tilde\psi_\la}\psi_\la \left .  \frac{d M_{4}^{(\la,l)} }{d l} \right|_{l=l_\la} dx=\tilde s_{20}s_{20}\left. \ds\f{d}{dl}\left(\int_\Omega  M_{4}^{(0,l)}   dx\right)\right|_{l=l_0}.
\end{split}
\end{equation}
It follows from Lemma \ref{suppst} and Eq. \eqref{suppM} that
\begin{equation*}
\left. \ds\f{d}{dl}\left(\int_\Omega e^{\alpha m(x)} M_{1}^{(0,l)}   dx\right)\right|_{l=l_0}=\mathcal S'(l_0)>0 \;\;\text{and}\;\; \left. \ds\f{d}{dl}\left(\int_\Omega  M_{4}^{(0,l)}   dx\right)\right|_{l=l_0}=0.
\end{equation*}
This, together \eqref{ptilde}, \eqref{ptilde2} and \eqref{limmulti}, yields
\begin{equation*}
 \lim_{\la\to 0} \frac{1}{\la} \left.\frac{ d \mathcal Re [\mu(l)]}{ d l}\right|_{l=l_\la}=\frac{\mathcal S'(l_0)}{2\int_\Omega e^{\alpha m(x)} dx}   >0.
 \end{equation*}
 This completes the proof.
\end{proof}

From Theorems \ref{thglobal}, \ref{the1},  \ref{sim} and \ref{the2},  we can obtain the following results on the dynamics of model \eqref{m4}, see also Fig. \ref{fig11}.
\begin{theorem} \label{ed}
Let $\left(u^{(\la,l)},v^{(\la,l)}\right)$ be the unique positive steady state (obtained in Theorem \ref{thglobal}) of model \eqref{m4} for
$l\in\mathcal L:=[\tilde l+\epsilon,1/\epsilon]$ and $\la\in (0,\delta_\epsilon]$ with $0<\epsilon\ll1$, where $\tilde l$ and $\delta_\epsilon$ are defined in Eq. \eqref{l_*} and Theorem \ref{thglobal}, respectively. Then the following statements hold.
\begin{enumerate}
\item [$\rm{(i)}$] If $\mathcal T(\alpha)<0$, where $\mathcal T (\alpha)$ is defined in Lemma \ref{suppst}, then there exists $\tilde\la_1\in(0,\delta_\epsilon)$ such that, for each $\la\in(0,\tilde\la_1]$, the positive steady state
$\left(u^{(\la,l)},v^{(\la,l)}\right)$ of model \eqref{m4} is locally asymptotically stable for $l\in\mathcal L$.
\item [$\rm{(ii)}$] If $\mathcal T(\alpha)>0$, then there exists $\tilde\la_2\in(0,\delta_\epsilon)$ and a continuously differentiable mapping
$$\la\mapsto l_\la:(0,\tilde\la_2]\to\mathcal L=[\tilde l+\epsilon,1/\epsilon]$$
such that, for each $\la\in(0,\tilde\la_2]$, the positive steady state
$\left(u^{(\la,l)},v^{(\la,l)}\right)$ of model \eqref{m4} is locally asymptotically stable when $l \in [ \tilde l+\epsilon,l_\la)$ and unstable when $l \in (l_\la, 1/\epsilon]$. Moreover, system \eqref{m4} undergoes a Hopf bifurcation at $\left(u^{(\la,l)},v^{(\la,l)}\right)$ when $l=l_\la$.
\end{enumerate}
\end{theorem}
\begin{proof}
To prove (i), we need to show that if $\mathcal T(\alpha)<0$, there exists $\tilde\la_1>0$ such that
\begin{equation*}
\sigma( A_{ l}(\la))\subset \{x+{\rm i}y : x,y\in\mathbb{R},x < 0  \} \; \text{for all}\;  \la\in(0,\tilde\la_1] \; \text{and}\;  l\in \mathcal{L}.
\end{equation*}
If it is not true, then there exists a sequence $\{  (\la_k ,l_k )\}_{k=1}^\infty$ such that $\lim_{k\to\infty}\la_k=0$, $\lim_{k\to\infty}l_k=l^*\in\mathcal L$, and
 \begin{equation*}
\sigma( A_{ l_k}(\la_k) )\not\subset \{x+{\rm i}y : x,y\in\mathbb{R},x < 0  \}.
\end{equation*}
Then, for $k\geq1$,
\begin{equation}\label{Aundl}
\left( A_{ l_k}(\la_k)-  \mu I \right) (\varphi,  \psi)^T= 0
\end{equation}
is solvable for some value of $(\mu_{k},\varphi_{k}, \psi_{k})$ with $\mathcal Re \mu_{k}\geq 0$ and $ (\varphi_{k}, \psi_{k})^T(\ne  (0,0)^T)\in (X_{\mathbb C})^2$.
Substituting $(\mu,\varphi,\psi)=(\mu_{k},\varphi_{k},\psi_{k})$ into \eqref{Aundl}, we have
\begin{equation*}
\begin{cases}
\ds\mu_{k}e^{\alpha m(x)} \varphi_{k}= L \varphi_{k}  +\la_ke^{\alpha m(x)} \left( M_1^{({\la_k}, l_k)} \varphi _{k}+ M_2^{({\la_k}, l_k)} \psi _{k} \right), \\
\ds\mu_{k} \psi_{k}= \theta\Delta \psi_{k}  + \la_k \left( M_3^{({\la_k}, l_k)}\varphi_{k} + M_4^{({\la_k}, l_k)} \psi _{k} \right).\\
\end{cases}
\end{equation*}
Ignoring a scalar factor, we see that $( \varphi_{k}, \psi_{k})^T(\neq ( 0,0)^T)\in (X_{\mathbb C})^2$ can be represented as
\begin{equation*}
\begin{cases}
 \varphi_{k} =\delta_{k}+ w_{k}, \;\;\text{where}\;\;\delta_{k}\geq0\;\;\text{and}\;\; w_{k}\in \left(X_{1}\right)_{\mathbb C},\\
 \psi_{k}  =(s_{1k}+{\rm i}s_{2k})+ z_{k},\;\;\text{where}\;\;s_{1k},s_{2k}\in\mathbb R\;\;\text{and}\;\; z_{k} \in \left(X_{1}\right)_{\mathbb C},\\
 \|\varphi_{k} \|_2^{2}+\|\psi_{k} \|_2^{2}=| \Omega |,
\end{cases}
\end{equation*}
and $(\mu_{k},\delta_{k}, s_{1k}, s_{2k}, w_{k},z_{k})$ satisfies
\begin{equation*}
\bm H(\mu_{k},\delta_{k}, s_{1k}, s_{2k}, w_{k},z_{k},l_k,\la_k)=(\mathcal H_1,\mathcal H_2,\mathcal H_3,\mathcal H_4,\mathcal H_5)^T=0,
\end{equation*}
where
\begin{equation*}
\begin{split}
\mathcal H_1:=& \int_\Omega  {  e^{\alpha m(x)} \left[ M_{1}^{(\la_k,l_k)}(\delta_{k}+w_{k})+ M_{2}^{(\la_k,l_k)} (s_{1k} + {\rm i} s_{2k}+z_k)  \right]} dx\\
&-\mu_{k}  \int_\Omega e^{\alpha m(x)} (\delta_{k}+w_{k})dx ,\\
\mathcal H_{2}:=&\ds  L w_k  +\la_k e^{\alpha m(x)}  \left[ M_{1}^{(\la_k,l_k)} (\delta_{k}+w_{k}) + M_{2}^{(\la_k,l_k)} (s_{1k} + {\rm i} s_{2k}+z_{k})  \right]\\
&\ds-\la_k \mu_{k} e^{\alpha m(x)}  (\delta_{k}+w_{k}) -\frac{\la_k }{ | \Omega |} \mathcal H_1,\\
\mathcal H_{3}:=& \ds \int_\Omega  {  \left[ M_{3}^{(\la_k,l_k)}(\delta_{k}+w_{k})+ M_{4}^{(\la_k,l_k)} (s_{1k}+ {\rm i} s_{2k}+z_{k})  \right]} dx\\
&-\mu_{k} (s_{1k} + {\rm i} s_{2k})| \Omega | ,\\
\mathcal H_{4}:=&\ds \theta\Delta z_k  + \la_k \left[ M_{3}^{(\la_k,l_k)} (\delta_{k}+w_{k}) +M_{4}^{(\la_k,l_k)} (s_{1k} + {\rm i} s_{2k}+z_{k}) \right]\\
&\ds-\la_k \mu_{k} (s_{1k} + {\rm i} s_{2k}+z_{k})  -\frac{\la_k }{ | \Omega |} \mathcal H_3,\\
\mathcal H_{5}:=&\ds  | \Omega |\left(\delta_{k}^2 +s_{1k}^2+s_{2k}^2-1\right) +\| w_{k}\|_2^2+\| z_{k}\|_2^2.
\end{split}
\end{equation*}
Using similar arguments as in the proof of Theorem \ref{the1}, we see that
\begin{equation*}
\lim_{k\to\infty} w_{k} = 0,\;\;\lim_{k\to\infty} z_{k} = 0\;\;\text{in}\;\;(X_1)_{\mathbb C}.
\end{equation*}
Since $\mathcal H_5(\mu_{k},\delta_{k}, s_{1k}, s_{2k}, w_{k},z_{k},l_k,\la_k)=0$, we see that, up to a subsequence, $\lim_{k\to\infty} \delta_{k } =\delta^*, \lim_{k\to\infty} s_{1k } =s_1^*$ and $\lim_{k\to\infty}s_{2k }=s_2^*$.
It follows from Lemma \ref{l3} that, up to a subsequence, $\lim_{k\to\infty} \mu_{k }=\mu^*$ with $\mathcal Re \mu^*\ge0$.
Then, taking $k\to\infty$ on both sides of $\mathcal H_j(\mu_{k},\delta_{k}, s_{1k}, s_{2k}, w_{k},z_{k},l_k ,\la_k)=0$ for $j=1,3$, we have
\begin{equation*}
\begin{cases}
\ds\mu^*\delta^*=  \delta^*\frac{\mathcal{S}(l^*)}{ \int_\Omega e^{\alpha m(x)}  dx } +  \left(s_1^*+{\rm i}s_2^*   \right)  \frac{\int_\Omega e^{\alpha m(x)} M_2^{(0,l^*)} dx}{ \int_\Omega e^{\alpha m(x)}  dx }  , \\
\ds\mu^* \left(s_1^*+{\rm i}s_2^*   \right)=  \frac{\delta^*}{ | \Omega |} \int_\Omega M_3^{(0,l^*)} dx+ \left(s_1^*+{\rm i}s_2^*   \right)\frac{1}{ | \Omega |} \int_\Omega M_4^{(0,l^*)} dx,
\end{cases}
\end{equation*}
where $\mathcal{S}(l)$ is defined in Lemma \ref{suppst}. Note from \eqref{M4} that
$\ds\int_\Omega  M_{4}^{(0,l^*)} dx=0$, and consequently, $\mu^*$ is an eigenvalue of
the following matrix
$$\left( {\begin{array}{cc}
\ds  \frac{\mathcal{S}(l^*)}{ \int_\Omega e^{\alpha m(x)}  dx } &\ds  \frac{\int_\Omega e^{\alpha m(x)} M_2^{(0, l^*)} dx}{ \int_\Omega e^{\alpha m(x)}  dx } \\
\ds \frac{1}{ | \Omega |} \int_\Omega M_3^{(0, l^*)} dx & 0
\end{array}} \right).$$
It follows from Lemma \ref{suppst} that $\mathcal{S}(l^*)<0$, which contradicts the fact that
$\mathcal Re \mu^*\geq 0$. Therefore, (i) holds.

Now we consider the case of $\mathcal T(\alpha)>0$. Then we only need to show that there exists $\tilde\la_2>0$ such that
\begin{equation*}
\sigma( A_{\tilde l+\epsilon}(\la))\subset \{x+{\rm i}y : x,y\in\mathbb{R},x < 0  \} \; \text{for}\;  \la\in(0,\tilde\la_2].
\end{equation*}
Note from Lemma \ref{suppst} that $\mathcal{S}(\tilde l+\epsilon)<0$. Then substituting $l_k= \tilde l+\epsilon$ in the proof of (i) and using similar arguments, we can also obtain a contradiction.
This proves (ii).
\end{proof}
\begin{remark}
We remark that, in Theorem  \ref{ed}, $\tilde\la_i$ depends on $\alpha$ and $\epsilon$ for $i=1,2$.
\end{remark}

\section{The effect of the advection}
In this section, we show the effect of advection. For later use, we first show the properties of the following auxiliary sequence:
\begin{equation}\label{B_k}
\left\{ B_k\right\}_{k=0}^\infty,\;\text{ where}\;B_k=\int_\Omega  m^k (x)(m(x)-1) dx.
\end{equation}
\begin{lemma}\label{Bk4}
Let $\{ B_k\}_{k=0}^\infty$ be defined in \eqref{B_k}, and let
$\mathcal B=\{x\in\Omega:\;m(x)>1\}$.
 Then $B_{k+1}\geq B_k$ for $k=0,1,2,\ldots$; $\lim_{k\to \infty} B_k =\infty$ if $\mathcal B\ne\emptyset$; and  $\lim_{k\to \infty} B_k =0$ if
$\mathcal B=\emptyset$.
\end{lemma}
\begin{proof}
A direct computation implies that
\begin{equation}\label{B_k1}
B_k=\int_\Omega  f_k (x) dx - \int_\Omega  g_k (x) dx,
\end{equation}
where
\begin{equation*}
 f_k (x) = m^k (x)(m(x)-1) \tilde I_1,\;\; g_k (x) = m^k (x)(1-m(x)) \tilde I_2,
\end{equation*}
and
\begin{equation*}
\tilde I_1 =
\begin{cases}
0,&x\in\Omega\setminus\mathcal B,\\
 1,&x\in\mathcal B,
\end{cases}\;\;\;\;\tilde I_2=
\begin{cases}
 1,&x\in\Omega\setminus\mathcal B,\\
 0,&x\in\mathcal B.
\end{cases}
\end{equation*}
Note that
\begin{equation*}
\begin{split}
&0\leq f_1 (x) \leq f_2 (x) \leq\ldots\leq f_k (x) \leq\ldots,\\
&g_1 (x) \geq g_2 (x)\geq\ldots\geq g_k (x) \geq\ldots\geq 0.
\end{split}
\end{equation*}
Then we obtain that $B_{k+1}\geq B_k$ for $k=0,1,2,\ldots$, $\lim_{k\to \infty} f_k (x)=f_* (x)$, and $\lim_{k\to \infty} g_k (x)=g_*(x)$, where $g_* (x)\equiv0$, and
\begin{equation*}
f_* (x) =
\begin{cases}
0,&x\in\Omega\setminus\mathcal B,\\
\infty,&x\in\mathcal B.
\end{cases}
\end{equation*}
Then we see from the Lebesgue's monotone convergence theorem that
\begin{equation}\label{kimfg}
\lim_{k\to\infty} \int_\Omega  f_k (x) dx =\int_\Omega f_* (x)dx \; \;\text{and}\;\; \lim_{k\to\infty} \int_\Omega  g_k (x) dx =0.
\end{equation}
Clearly, $\int_\Omega f_* (x)dx=\infty$ if $\mathcal B\ne\emptyset$, and $\int_\Omega f_* (x)dx=0$ if $\mathcal B=\emptyset$.
This, combined with \eqref{B_k1} and \eqref{kimfg}, implies that
$\lim_{k\to \infty} B_k =\infty$ if $\mathcal B\ne\emptyset$, and  $\lim_{k\to \infty} B_k =0$ if
$\mathcal B=\emptyset$.
\end{proof}
Now, we can consider function $\mathcal T(\alpha)$, which determines the existence of Hopf bifurcation from Theorem \ref{ed}.
\begin{theorem}\label{lHopf}
Let $\mathcal T (\alpha)$ and $\mathcal B$ be defined in Lemmas \ref{suppst} and \ref{Bk4}, respectively. Then the following statements hold:
\begin{enumerate}
\item [$\rm{(i)}$] If $\mathcal T(0)=\int_\Omega ( m(x) -1) dx\ge0$, then $\mathcal T(\alpha)>0$ for any $\alpha>0$;
\item [$\rm{(ii)}$] If $\mathcal T(0)<0$ and $\mathcal B\ne\emptyset$, then there exists $\alpha_*>0$ such that $\mathcal T(\alpha_*)=0$, $\mathcal T(\alpha)<0$ for $0<\alpha<\alpha_*$, and $\mathcal T(\alpha)>0$ for $\alpha>\alpha_*$;
\item [$\rm{(iii)}$] If $\mathcal B=\emptyset$, then $\mathcal T(\alpha)\le0$ for any $\alpha\geq0$.
\end{enumerate}
\end{theorem}

\begin{proof}
For simplicity, we denote
\begin{equation*}
\mathcal T^{(k)} (\alpha)=\frac{d^k \mathcal T(\alpha)}{d \alpha^k}\;\;\text{for}\;\;k\geq1\;\;\text{and}\;\; \mathcal T^{(0)} (\alpha)=\mathcal T(\alpha).
\end{equation*}
A direct computation yields
\begin{equation*}
\mathcal T^{(k)} (\alpha)=\int_\Omega e^{\alpha m(x)}m^k (x)( m(x) -1) dx\;\; \text{for}\;\;k\geq0.
\end{equation*}
Note that $m(x)$ is non-constant. Then we see that, for $k\ge0$,
\begin{equation*}
\mathcal T^{(k+1)} (\alpha)-\mathcal T^{(k)} (\alpha)=\int_\Omega e^{\alpha m(x)}m^k (x)( m(x) -1)^2 dx>0,
\end{equation*}
which yields
\begin{equation*}
\frac{d\left[e^{-\alpha} \mathcal T^{(k)} (\alpha)    \right]}{d \alpha}=e^{-\alpha} \mathcal T^{(k+1)} (\alpha) - e^{-\alpha} \mathcal T^{(k)} (\alpha) >0,
\end{equation*}
and consequently,
\begin{equation}\label{mTae0}
\mathcal T^{(k)} (\alpha) > e^{\alpha} \mathcal T^{(k)} (0)\;\;\text{for all}\;\;\alpha>0.
\end{equation}
Here $\mathcal T^{(k)} (0) =B_k$, where $B_k$ is defined in \eqref{B_k}.

Now we show that (i) holds. Note that $\mathcal T (0) =\mathcal T^{(0)} (0) = B_0\ge0$. Then we see from \eqref{mTae0} that $\mathcal T (\alpha)>0$ for all $\alpha>0$. Then we show that (iii) holds. Since $\mathcal B=\emptyset$, we have $0\le m(x)\le1$, and consequently, $\mathcal T(\alpha)\le0$ for all $\alpha\ge0$.

Finally, we consider (ii). Note that $\mathcal T (0) = B_0<0$. It follows from Lemma \ref{Bk4} that there exists an integer $k_*\geq1$ such that $B_k\geq0$ for $k\geq k_*$ and
$B_k<0$ for $0\le k< k_*$.
This, combined with \eqref{mTae0}, implies that
\begin{equation}\label{Ta1}
\mathcal T^{(k)} (\alpha) >0\;\text{for all}\;\alpha>0\;\text{and}\;k\geq k_*.
\end{equation}
Then $\mathcal T^{(k_*-1)} (\alpha)$ is strictly increasing for $\alpha>0$, and consequently,
$$ \lim_{\alpha\to\infty} \mathcal T^{(k_*-1)} (\alpha)=\alpha^\infty_{k_*-1}. $$
We claim that $\alpha^\infty_{k_*-1}=\infty$. If it is not true, then
\begin{equation}\label{Ta2}
\lim_{\alpha\to\infty} \frac{\mathcal T^{(k_*-1)} (\alpha)}{\alpha}=0.
\end{equation}
Note from \eqref{Ta1} that $\mathcal T^{(k_*)} (\alpha)$ is also strictly increasing for $\alpha>0$. This, combined with the fact $\mathcal T^{(k_*)} (0)=B_{k_*}\geq0$, implies that
$$\lim_{\alpha\to\infty} \mathcal T^{(k_*)} (\alpha)>0.$$
Then we see from the L'Hospital's rule that
$$\lim_{\alpha\to\infty} \frac{\mathcal T^{(k_*-1)} (\alpha)}{\alpha}=\lim_{\alpha\to\infty} \mathcal T^{(k_*)} (\alpha)>0,$$
which contradicts \eqref{Ta2}. Therefore, the claim is true and $\lim_{\alpha\to\infty} \mathcal T^{(k_*-1)} (\alpha)=\infty$. This, combined with the fact that $\mathcal T^{(k_*-1)}(0)=B_{k_*-1}$, implies that there exists $\alpha_{k_*-1}$ such that $\mathcal T^{(k_*-1)}(\alpha_{k_*-1} )=0$, and
$$\mathcal T^{(k_*-1)}(\alpha)<0\;\text{ for}\;\alpha\in[0,\alpha_{k_*-1})\;\;\text{ and}\; \;\mathcal T^{(k_*-1)}(\alpha)>0\;\text{ for}\;\alpha>\alpha_{k_*-1}.$$
This implies that
$\mathcal T^{(k_*-2)}(\alpha )$ is strictly decreasing for $\alpha\in[0,\alpha_{k_*-1}]$ and strictly increasing for $\alpha\in[\alpha_{k_*-1},\infty)$.
Therefore, $ \lim_{\alpha\to\infty} \mathcal T^{(k_*-2)} (\alpha)=\alpha^\infty_{k_*-2}$.
Then we claim that $\alpha^\infty_{k_*-2}=\infty$. If it is not true, then
\begin{equation}\label{Ta22}
\lim_{\alpha\to\infty} \frac{\mathcal T^{(k_*-2)} (\alpha)}{\alpha}=0.
\end{equation}
By the L'Hospital's rule again, we have
$$\lim_{\alpha\to\infty} \frac{\mathcal T^{(k_*-2)} (\alpha)}{\alpha}=\lim_{\alpha\to\infty} \mathcal T^{(k_*-1)} (\alpha)=\infty,$$
which contradicts \eqref{Ta22}. Therefore, $ \lim_{\alpha\to\infty} \mathcal T^{(k_*-2)} (\alpha)=\infty$. Then, there exist $\alpha_{k_*-2}$ such that $\mathcal T^{(k_*-2)}(\alpha_{k_*-2} )=0$, and
$$\mathcal T^{(k_*-2)}(\alpha)<0\;\text{ for}\;\alpha\in[0,\alpha_{k_*-2})\;\;\text{ and}\; \;\mathcal T^{(k_*-2)}(\alpha)>0\;\text{ for}\;\alpha>\alpha_{k_*-2}.$$
Therefore, we can repeat the previous arguments to obtain (ii). This completes the proof.
\end{proof}
Then, by virtue of Theorem \ref{lHopf}, we show the effect of advection rate $\alpha$ on the occurrence of Hopf bifurcations for model \eqref{m4}.

\begin{proposition}\label{effad}
Assume that $\mathcal T(0)<0$, $\mathcal B\ne\emptyset$, and let $\alpha_*$ be defined in Theorem \ref{lHopf}. Then for any $\epsilon$ with $0<\epsilon\ll1$ and $\alpha\neq\alpha_*$, there exists $\tilde\la(\alpha,\epsilon)>0$ such that the following statements hold.
\begin{enumerate}
\item [$\rm{(i)}$] If $\alpha<\alpha_*$, then for each $\la\in(0,\tilde\la(\alpha,\epsilon)]$, the positive steady state
$\left(u^{(\la,l)},v^{(\la,l)}\right)$ of model \eqref{m4} is locally asymptotically stable for $l\in\mathcal L:=[\tilde l+\epsilon,1/\epsilon]$, where $\tilde l$ is defined in \eqref{l_*} and depends on $\alpha$.
\item [$\rm{(ii)}$] If $\alpha>\alpha_*$, then there exists a continuously differentiable mapping
$$\la\mapsto l_\la:[0,\tilde\la(\alpha,\epsilon)]\to\mathcal L=[\tilde l+\epsilon,1/\epsilon]$$
such that, for each $\la\in (0,\tilde\la(\alpha,\epsilon)]$, the positive steady state
$\left(u^{(\la,l)},v^{(\la,l)}\right)$ of model \eqref{m4} is locally asymptotically stable for $l \in [ \tilde l+\epsilon,l_\la)$ and unstable for $l \in (l_\la, 1/\epsilon]$, where $\tilde l$ is defined in \eqref{l_*} and depends on $\alpha$. Moreover, system \eqref{m4} undergoes a Hopf bifurcation at $\left(u^{(\la,l)},v^{(\la,l)}\right)$ when $l=l_\la$.
\end{enumerate}
\end{proposition}

By Proposition \ref{effad}, we see that the advection rate affects the occurrence of Hopf bifurcations if $\mathcal T(0)<0$ and $\mathcal B\ne\emptyset$. Actually, there exists a critical value $\alpha_*$ such that
Hopf bifurcation can occur (or respectively, cannot occur) with $\alpha>\alpha_*$ (or respectively, $\alpha<\alpha_*$).
Next, we show that the advection rate can also affect the values of Hopf bifurcations if $\mathcal T(0)>0$.

\begin{proposition}\label{last}
Assume that $\mathcal T(0)>0$. Then for any $\epsilon$ with $0<\epsilon\ll1$ and $\alpha\ge0$, there exists $\tilde\la(\alpha,\epsilon)>0$  and  a continuously differentiable mapping
$$\la\mapsto l_\la:[0,\tilde\la(\alpha,\epsilon)]\to\mathcal L=[\tilde l+\epsilon,1/\epsilon]$$
such that, for each $\la\in (0,\tilde\la(\alpha,\epsilon)]$, the positive steady state
$\left(u^{(\la,l)},v^{(\la,l)}\right)$ of model \eqref{m4} is locally asymptotically stable for $l \in [ \tilde l+\epsilon,l_\la)$ and unstable for $l \in (l_\la, 1/\epsilon]$, and system \eqref{m4} undergoes a Hopf bifurcation at $\left(u^{(\la,l)},v^{(\la,l)}\right)$ when $l=l_{\la}$, where $\tilde l$ is defined in \eqref{l_*} and depends on $\alpha$. Moreover, $\lim_{\la\to0}l_{\la}=l_0$, and $l_0$ (defined in Lemma \ref{suppst}) depends on $\alpha$ and satisfies the following properties:
\begin{enumerate}
\item [$\rm{(i)}$] If $\mathcal H>0$, then $l'_0(\alpha)|_{\alpha=0}>0$, and $l_0(\alpha) $ is strictly increasing for
    $\alpha\in(0,\epsilon)$ with $0<\epsilon\ll1$;
\item [$\rm{(ii)}$] If $\mathcal H<0$, then $l'_0(\alpha)|_{\alpha=0}<0$, and $l_0(\alpha)$ is strictly decreasing for
    $\alpha\in(0,\epsilon)$ with $0<\epsilon\ll1$.
\end{enumerate}
Here $\mathcal H=2\left(\ds\int_{\Omega}m(x)dx\right)^2-|\Omega|\ds\int_\Omega m(x) dx-|\Omega|\ds\int_{\Omega}m^2(x)dx$.
\end{proposition}

\begin{proof}
Let $\mathcal S_1(c), \mathcal S'_1(c),\mathcal S_3(c), c_0$ be defined in the proof of Lemma \ref{suppst}, where
$'$ is the derivative with respect to $c$, and they all depend on $\alpha$. Therefore, we denote them by
$\mathcal S_1(c,\alpha), \ds\frac{\partial\mathcal S_1}{\partial c}(c,\alpha),\mathcal S_3(c,\alpha), c_0(\alpha)$, respectively. By  \eqref{cc0l} and \eqref{coadv}, we see that $c_0(\alpha)=c_{0l}$ with $l=l_{0}(\alpha)$ and $\ds\frac{d c_{0l}}{dl}<0$, which implies that $l'_0(\alpha)$ has the same sign as $-c'_0(\alpha)$.

Since $\mathcal T(0)>0$, it follows from  Theorem \ref{lHopf} that $\mathcal T(\alpha)>0$ for all $\alpha\ge0$. This combined with Lemma \ref{suppst} implies that $c_0(\alpha)$ exists for all $\alpha\ge0$.
From the proof of Lemma \ref{suppst}, we see that
\begin{equation}\label{s30}
 \mathcal S_3(c_0(\alpha),\alpha)=0,
\end{equation}
and $\ds\frac{\partial \mathcal S_3}{\partial c}>0$ for all $\alpha\ge0$. Therefore, $-c'_0(\alpha)$ has the same sign as $\ds\frac{\partial \mathcal S_3}{\partial \alpha}(c_0(\alpha),\alpha)$.
By \eqref{s3c} and a direct computation yields
\begin{equation}\label{ca00}
\begin{split}
 \f{\partial \mathcal S_3}{\partial\alpha }=&\left(\mathcal S_1\ds\f{\partial^2\mathcal S_1}{\partial c\partial\alpha} -   \f{\partial\mathcal S_1}{\partial\alpha} \f{\partial\mathcal S_1}{\partial c}\right) \f{ 1 }    {\mathcal S^2_1}\int_\Omega \left[e^{\alpha m(x)} m(x)- c e^{2 \alpha m(x)}\right] dx\\
 &+ \f{\partial\mathcal S_1}{\partial c}\f{1}{\mathcal S_1} \int_\Omega \left[e^{\alpha m(x)} m^2(x)- 2c e^{2 \alpha m(x)} m(x)\right] dx+2 \int_\Omega  e^{2 \alpha m(x)} m(x) dx,
\end{split}
\end{equation}
where
$$
\f{\partial\mathcal S_1}{\partial\alpha}=\int_\Omega  \frac{ e^{\alpha m(x)}m(x) }{ \left(1+ c   e^{\alpha m(x)} \right)^2 } dx,    \;\;\ds\f{\partial^2\mathcal S_1}{\partial c\partial\alpha}=-2\int_\Omega  \frac{ e^{2\alpha m(x)}m(x) }{ \left(1+ c   e^{\alpha m(x)} \right)^3 } dx.
$$
Substituting $\alpha=0$ into \eqref{s30}, we obtain that
\begin{equation}\label{c0lst}
c_0(0)=\frac{1}{2 |\Omega|}\int_{\Omega}(m(x)-1)dx.
\end{equation}
Then plugging $c=c_0(0)$ and $\alpha=0$ into \eqref{ca00}, we have
$$\f{\partial \mathcal S_3}{\partial\alpha }(c_0(0),0)=\f{ 2\left[\left (\int_\Omega m(x) dx\right)^2-|\Omega| \int_\Omega m(x) dx-|\Omega| \int_\Omega m^2(x) dx \right]}{\int_\Omega m(x) dx+ |\Omega| },$$
which implies that $l_0(\alpha)$ satisfies (i) and (ii).
\end{proof}
Here we only show the effects of advection rate on the
values of Hopf bifurcations for $0<\alpha\ll1$, and the general case still awaits further investigation.

 \vspace{0.3in}
\noindent {\bf Data availability statement}

All data generated or analyzed during this study are included in this manuscript.

\end{document}